\documentclass[10pt]{article}
\usepackage{amsfonts, amsthm, amsgen, amsmath, amssymb}
\usepackage{hyperref, enumerate}

\newtheorem{theorem}{Theorem}[subsection]

\newtheorem{lemma}[theorem]{Lemma}
\newtheorem{corollary}[theorem]{Corollary}
\newtheorem{proposition}[theorem]{Proposition}
\theoremstyle{definition}
\newtheorem{definition}[theorem]{Definition}

\theoremstyle{remark}
\newtheorem{remark}[theorem]{Remark}
\newtheorem{remarks}[theorem]{Remarks}

\numberwithin{equation}{section}
\usepackage{amscd}
\usepackage{xypic}
\usepackage{amsmath, amssymb}

\numberwithin{equation}{subsection}
\newcommand{\be}%
  {\protect\setcounter{equation}{\value{subsubsection}}}
  \newcommand{\ee}%
   {\protect\setcounter{subsubsection}{\value{equation}}}

  {\protect\setcounter{subsubsection}{\value{equation}}}

\newcommand{\oSpec}{\operatorname{Spec}}

\newcommand{\bG}{\mathbf G}
\newcommand{\frG}{\mathfrak G}

\newcommand{\bX}{\mathbf X}
\newcommand{\bY}{\mathbf Y}
\newcommand{\bZ}{\mathbf Z}

\newcommand{\C}{\rm C}
\newcommand{\cN}{\mathcal N}

\newcommand{\cX}{\mathcal X}

\newcommand{\GL}{\mathbf {GL}}

\newcommand{\oG}{\overline{G}}

\renewcommand{\b}{\mathbf{b}}

\newcommand{\A}{\mathcal A}

\newcommand{\D}{\mathcal D}
\renewcommand{\L}{\mathcal L}
\newcommand{\E}{\mathcal E}

\newcommand{\U}{\mathcal U}

\newcommand{\V}{\mathcal V}

\newcommand{\X}{\mathcal X}
\newcommand{\Y}{\mathcal Y}
\newcommand{\Z}{\mathcal Z}

\newcommand{\NN}{\mathbb N}

\newcommand{\K}{\mathcal{K}}

\renewcommand{\O}{\mathcal{O}}

\renewcommand{\a}{\mathbf a} 

\def\displaytimes_#1{\mathrel{\mathop{\times}\limits_{#1}}}

\def\displayotimes_#1{\mathrel{\mathop{\bigotimes}\limits_{#1}}}

 \def\ari[#1]{\ar@{^(->}[#1]}
 \def\are[#1]{\ar[#1]^{\txt{\'et}}}
 \def\areh[#1]{\ar[#1]|{\txt{$H$-eq}}^{\txt{\'et}}}
 \def\ars[#1]{\ar@{->>}[#1]}
 \newcommand{\dplus}{\ar@{}[d]|{\mbox{$\oplus$}}}
 \newcommand{\dtimes}{\ar@{}[d]|{\mbox{$\times$}}}

\def \A{{\mathcal A}}
\def \bA{\mathbb A}

\def \a{\it a}

\def \B{\mathcal B}

\def \rmB{\rm B}
\def \b{\it b}

\def \C{\mathcal C}

\def \Cl{\mathbb C}

\def \colimm{\underset {m \rightarrow \infty}  {\hbox {lim}}}
\def \colimbm{\underset {(m_1, \cdots, m_k) \rightarrow \infty}  {\hbox {lim}}}

\def \colimi{\underset {i \rightarrow \infty}  {\hbox {lim}}}
\def \colimb{\underset {b \epsilon {\mathcal A}}  {\hbox {colim}}}
\def \colimA{\underset {\a \epsilon {\mathcal A}}  {\hbox {colim}}}
\def \colimB{\underset {b \epsilon {\mathcal B}}  {\hbox {colim}}}

\def \colimalpha{\underset {\alpha}  {\hbox {colim}}}
\def \colima{\underset {\a \eps \A}  {\hbox {colim}}}

\def \colimK.{\underset {\underset K^.  \rightarrow}  {\hbox {lim}}}

\def \colimU.{\underset {\underset U_.  \rightarrow}  {\hbox {lim}}}
\def \colimV{\underset {\underset V \rightarrow}  {\hbox {lim}}}

\def \compl{\, \, {\widehat {}}}

\def \D{\mathcal D}

\def \E{{\rm E}}
\def \rmE{\rm E}
\def \EG{\rm EG}

\def \EG1{E{(G \times {\mathbb C}^*)}{\underset {G\times {\mathbb C}^*} 
\times}}

\def \EZ(s)1{E{(Z(s) \times {\mathbb C}^*)}{\underset {(Z(s)\times {\mathbb
C}^*)}  \times}}

\newcommand{\eps}{ \, {\boldsymbol\varepsilon} \,}

\def \EM(u){EM(u){\underset {M(u)}  \times}}
\def \EM(us){EM(u,s){\underset {M(u, s)}  \times}}

\newcommand{\oF}{\operatorname{F}}
\def \rmF{\rm F}
\def \f{\it f}

\def \bG{\mathbf G}
\def \rmG{\rm G}
\def \G{{\mathcal G}}
\def \GL{\rm {GL}}

\def \oG{\rm G}

\def\holimD{\mathop{\textrm{holim}}\limits_{\Delta }}

\def \holimm {\underset {\infty \leftarrow m}  {\hbox {holim}}}
\def \holimbm {\underset {\infty \leftarrow (m_1, \cdots, m_k)}  {\hbox {holim}}}

\def \holims {\underset {\underset {s \eps \K}  \leftarrow }  {\hbox {holim}}}
\def \holimt {\underset {\underset {t \eps \K}  \leftarrow }  {\hbox {holim}}}

\def \H{\mathbb H}

\def \rmH{\rm H}

\def \h{\it h}

\def \invlim1{\underset {\infty \leftarrow q}  {\hbox {lim}}^1}
\def \in{\mbox{ {\rm {in}} }}

\def \K{\mathcal K}
\def \bK{\mathbf K}
\def \bK{\mathbf K}
\newcommand{\ok}{\operatorname{k}}
\newcommand{\oK}{\operatorname{K}}

\def \L3{\Lambda \times \Lambda \times \Lambda}
\def \L2{\Lambda \times \Lambda}

\def \limalpha{\underset {\infty \leftarrow {\alpha}}  {\hbox {lim}}}
\def \lim{\underset \leftarrow  {\hbox {lim}}}

\def \longright2arrow{{\overset \longrightarrow  {\overset {} 
\longrightarrow}}}

\def \L{L\times \Cl ^*}

\def \Map{\underline {Map}}
\def \Map{{\mathcal M}ap}

\def \bm{\mathbf m}

\def \rmM{\rm M}

\def \N{\mathcal N}
\def \NN{\mathbb N}
\def \ndots{(n_1, \cdots, n_k)}
\def \bn{\mathbf n}
\def \rmN{\rm N}

\def \O{{\mathcal O}}

\def \P{\mathbb P}
\def \rmP{\rm P}
\newcommand{\p}{\mathfrak p}

\def \Spt{\rm {Spt}}

\def \ra{\rightarrow}

\def \RG^{R(G)^{\hat {}}\ }

\def \res{respectively}

\def \rmR{\rm R}

\def \Sm{\rm {Sm}}
\def \rmS{\rm S}

\def \\oSpeck{{\hbox {\oSpec}\ \ k}}
\def \o\oSpec{{\rm {\oSpec}}\, }

\def\Spt{\rm {\bf Spt}}

\newcommand{\Psh}{\operatorname{Psh}}

\newcommand{\SPrsh}{\operatorname{SPrsh}}

\newcommand{\Sch}{\operatorname{Sch}}
\newcommand{\Sh}{\operatorname{Sh}}
\def \bs{\mathbf s}
\def \mbS{\mathbb S}

\def \topGcoh*{^{top, *} _{G}}
\def \topGho*{ _{top,*} ^{G}}

\newcommand{\Top}{\operatorname{Top}}

\def \U{\mathcal U}

\def \rmU{\rm U}

\def \V{\mathcal V}

\def \rmV{\rm V}

\def \rmW{\rm W}

\def \cX{\mathcal X}

\def \bX{\mathbf X}
\def \rmX{\rm X}

\def \cY{\mathcal Y}
\def \rmY{\rm Y}

\def \Z(s){Z(s) \times {\mathbb C}^*}
\def \Z{\mathbb Z}
\def \cZ{\mathcal Z}

\def \rmZ{\rm Z}

\begin{document}

\title{Algebraic K-theory of geometric fields: rigidity}
\author{Gunnar Carlsson\\
Department of Mathematics, \\
Stanford University, Stanford, California, 94305, USA \renewcommand{\footnotemark}{}%
  \footnote{The first author was partially supported by NSF DMS-0406992.}\\
\texttt{gunnar@math.stanford.edu \ {and}}\\
{Roy Joshua}\\
Department of Mathematics, \\
Ohio State University, Columbus, Ohio, 43210, USA. \renewcommand{\footnotemark}{}%
  \footnote{The second author was partially supported by NSF DMS-1200284 and by the Simons Foundation MP-TSM-00002471.}%
\\
\texttt{joshua@math.ohio-state.edu}\\
{}\\
\\
\texttt{}}
\date{}
\maketitle

\begin{abstract}
In this paper we consider the problem of Galois descent for suitably completed algebraic K-theory of fields.
One of the main results in the present paper is a suitable form of rigidity for Borel-style  generalized equivariant cohomology with
respect to certain spectra, which includes mod-${\ell}$ Algebraic K-theory, for $\ell$ -invertible in the base field.
Such a rigidity theorem in fact provides a determination of the mod-$\ell$ Algebraic K-theory of geometric fields, i.e., fields that 
contain an algebraically closed sub-field in terms of the absolute Galois group of the bigger field and proves one conjecture due 
to the first author.
\vskip .2cm
What remains in a proof of a related conjecture of the first author (see \cite{C13}) is a derived completion theorem for equivariant algebraic K-theory 
with respect to the action of profinite groups. Such a derived completion theorem, with certain restrictions, is proven in
the  paper, \cite{CJ26-2}. What we do in the remaining part of this paper is to
provide a comparison of the generalized Borel style equivariant K-theory considered in \cite{CJ26-2} and also in \cite{C13}, with the generalized Borel style
equivariant K-theory defined using more standard models of geometric classifying spaces. This enables
us to apply the rigidity results proven in this paper to the generalized Borel style equivariant K-theory considered in \cite{CJ26-2}, thereby enabling us to provide an affirmative solution to
the above mentioned conjecture in many cases.

\end{abstract}
\maketitle

\vfill \eject
\section{ Introduction}
\label{intro}
The main goal of this paper and its sequel \cite{CJ26-2} is to provide a computation of the mod-$\ell^n$ algebraic K-theory of geometric fields, i.e., fields $\oF$ that contain
an algebraically closed sub-field $\ok$, where $\ell$ is a prime invertible in the field. A conjecture due to the first author (see \cite[Conjecture 1.3]{C11}) says that this should be given by the 
mod-$\ell^n$ (Atiyah-Segal-style)
equivariant algebraic K-theory of the algebraically closed sub-field $\ok$, equivariant with respect to the (trivial) action of the absolute Galois
group $\oG_{\oF}$ of the bigger field $\oF$. In fact, this extends a well-known theorem of Suslin (see \cite{Sus}) which says that the mod-$\ell^n$-algebraic K-theory
of two algebraically closed fields are isomorphic and is a consequence of what is now known as the rigidity property of mod-$\ell^n$ algebraic K-theory.
The main result of this paper combined with the main result of \cite{CJ26-2} is an affirmative solution of this conjecture. 
\vskip .2cm
Our approach to a solution of this conjecture  is to break the problem into two parts: the goal of the first part is
to prove the weak-equivalence:
\be \begin{equation}
 \label{step.1}
\bK(\oSpec \, \oF) \compl_{\ell} \simeq  \bK(\rmB\oG_{\oF})\compl_{\ell} 
\end{equation} \ee
\vskip .1cm \noindent
where $\compl_{\ell}$ denotes the derived completion at the prime $\ell$ as in \cite[(3.0.6)]{CJ23}, $\rmB\oG_{\oF}$
is the classifying space of the absolute Galois group $\oG_{\oF}$ suitably defined, and $\bK$ denotes the spectrum representing algebraic K-theory. (In fact, we make use of the
geometric classifying spaces commonly used
in the motivic contexts for linear algebraic groups (see, for example, \cite{MV} or \cite{Tot}) and extend them to profinite groups.)
This is in fact one of the main goals
of the present paper, and we prove this with no restrictions on the absolute Galois group $\oG_{\oF}$, or on the fields $\ok$ and $\oF$, except
that $\ok$ is an algebraically closed subfield of $\oF$. Clearly a good example of this would be if
$\oF$ is the function field of a scheme defined over the algebraically closed field $\ok$.
\vskip .2cm
The way we prove ~\eqref{step.1} is as a corollary to 
a rigidity theorem for Borel-style generalized equivariant cohomology with respect to spectra
which are $\ell$-primary torsion for a fixed prime $\ell$ different from the characteristic of the base field. 
A key aspect of our proof of such a rigidity theorem is finding a suitable model for the classifying spaces of 
profinite groups: we accomplish this by imbedding a profinite group into an inverse system of linear algebraic groups
and adapting standard models of geometric classifying spaces for such linear algebraic groups.
This rigidity theorem may be viewed as an extension of well-known rigidity theorems in the non-equivariant
setting. The strategy of the proof is by doing a reduction to the non-equivariant case by means
of a sheaf-theoretic argument somewhat similar to using a  Leray-spectral sequence. 
The main result here is Theorem ~\ref{main.thm1.1} which is also restated as Theorem ~\ref{main.thm.1}.
\vskip .2cm
The remaining aspect of the conjecture due to the first author (see \cite[Conjecture 1.3]{C11}) is the relationship between the mod-$\ell^n$ Borel-style equivariant algebraic K-theory with the
corresponding mod-$\ell^n$ Atiyah-Segal style  equivariant Algebraic K-theory, equivariant with respect to the action of the absolute Galois group. This is
given by a derived completion theorem for actions of such profinite groups similar to the derived completion theorem for actions of linear algebraic groups proved
in \cite{CJ23}. In other words, 
the goal of the second part is to prove the weak-equivalence:
\be \begin{equation}
   \label{step.2}
\bK(\rmB\oG_{\oF})\compl_{\ell}  \simeq \bK(\oSpec \, \ok, \oG_{\oF}) \compl_{\rho_{\rmG, \ell}} 
 \end{equation} \ee
where $\rho_{\rmG, \ell}: \bK(\oSpec \, \ok, \oG) \ra \bK(\oSpec \, \ok) \compl_{\ell}$ is the obvious map. This we achieve 
in \cite[Theorem 1.2]{CJ26-2} for a different model of the classifying space $\rmB\oG_{\oF}$ under some rather mild restrictions on $\oF$.
\vskip .3cm
In \cite{C11}, the first author has already established a form of this derived Atiyah-Segal completion theorem, under
certain restrictions, for example, when $\oG_{\oF}$ is abelian.  The  Borel-style generalized equivariant cohomology defined in \cite{C13} and \cite{CJ26-2} makes use of different models
for the classifying spaces of profinite groups, i.e., different from the geometric classifying spaces used in ~\eqref{step.1} and
~\eqref{step.2}. 
Therefore, a second main result in the present paper is to relate these
two forms of  Borel-style generalized equivariant cohomology for actions of profinite groups. 
This appears in Theorem ~\ref{main.thm.2}. With this result in place, the rigidity theorem in Theorem ~\ref{main.thm.1} 
provides a rigidity theorem for the  Borel-style generalized equivariant cohomology considered in \cite{CJ26-2}.
Combining these two steps together with \cite[Theorem 1.2]{CJ26-2}, provides a proof of both parts of the conjecture under certain restrictions on the absolute Galois group $\oG_{\oF}$. 
\vskip .3cm
Here is a brief outline of the paper. We discuss the basic framework and the main results in the rest of this section. The second section is devoted
to a discussion of Borel-style generalized equivariant cohomology theories defined with respect to spectra and for pro-group actions. Since this is carried out using
geometric models for the classifying spaces of algebraic groups, we quickly 
recall the framework of {\it admissible gadgets} as in 
\cite{MV}
and  generalized equivariant cohomology theories  for the action of
a single algebraic group. (This is discussed in detail in \cite{CJ23}.) Then we discuss actions of pro-groups. We prove that the generalized equivariant cohomology theories so defined
are independent of the choice of an inverse system of geometric classifying spaces: this proof makes use of the slice tower
as in \cite{Voev2}.
\vskip .3cm
The third section is devoted to a discussion of rigidity, where we establish rigidity for Borel-style generalized equivariant cohomology theories,
the main result being  Theorem ~\ref{main.thm.1}. One may view this theorem as an equivariant version of
well-known rigidity theorems: see \cite{Sus}. 
In fact the proof is by reducing rigidity for such equivariant cohomology theories to rigidity for the corresponding non-equivariant cohomology theories.
The fourth section is devoted to a comparison with the generalized equivariant cohomology theories defined using a tower construction as in \cite{CJ26-2}.
With such a comparison theorem, we obtain rigidity for the generalized equivariant cohomology theories defined using the tower construction.

\subsection{\bf Statement of results}
We adopt standard conventions on  motivic homotopy theory. We fix a base scheme, which for the most part will be  
a {\it perfect infinite} field $\ok$, and 
restrict to the
category of smooth schemes over this base scheme. For the most part, we will restrict to the case where the 
field $\ok$ is in fact algebraically closed.
But some of our constructions and intermediate results hold in the above more general framework.
\vskip .3cm
The latter category will be denoted $\Sm/\ok$. For the most part the site we put on $\Sm/\ok$ will be the Nisnevich site,
which will be denoted $\Sm/\ok_{Nis}$. The unstable model structure on pointed simplicial presheaves on this site, will be the {\it local projective} model structure 
of \cite{Bl} where the generating 
cofibrations are injective maps of simplicial presheaves of the form $\Lambda [n]_+ \wedge \rmU_+ \ra \Delta [n]_+ \wedge \rmU_+$ and
$\delta \Delta[n]_+ \wedge \rmU_+ \ra \Delta[n]_+ \wedge \rmU_+$ where $\Delta[n]$, $\Lambda [n]$ and $\delta \Delta [n]$ are the obvious
simplicial sets and $\rmU$ is an object of the chosen site.  In the local projective model structure, the weak-equivalences will be stalk-wise weak-equivalences
and fibrations will be characterized by the right-lifting property with respect to trivial cofibrations. The 
 ${\mathbb A}^1$-model structure may be obtained 
by localizing the local projective model structure with respect to  ${\mathbb A}^1$-equivalences: see  \cite{Hov-2} (or \cite{Dund})
 for more details. 
\vskip .3cm
 One starts with the 
${\mathbb A}^1$-localized local projective model category of
 pointed simplicial presheaves on the Nisnevich site on $\Sm/\ok$. Then an $\rmS^1$-spectrum in this category is a
system of pointed simplicial presheaves $\{\E_n|n \ge 0\}$ equipped with maps $\rmS^1 \wedge \E_n \ra E_{n+1}$ which 
are compatible as $n$ varies. We will let $\Spt_{{\rm S}^1}(\ok)$ denote the category of such ${\rm S}^1$-spectra. 
This is a symmetric monoidal category with respect to the operation of smash product, denoted $\wedge$: ${\mathbb S}$,
the sphere spectrum, will denote the unit of this symmetric monoidal category.
We will provide $\Spt_{{\rm S}^1}(\ok)$ with the {\it stable} model structure
as discussed in \cite[section 5]{Hov-2}.
The presheaves of spectra $\E$ that we consider in the paper will always be assumed to satisfy the following  hypotheses (see 
\cite[section 2.1]{Lev} and also \cite[Definition 1.2]{Yag}):
\subsubsection{}
\label{spectra.props}
\begin{enumerate}  [\rm(i)]
\item $\E$ is $N$-connected for some integer $N$. (The $N$ usually will be $-1$.)
     \item $\E$ is homotopy invariant, i.e., for each smooth scheme $\rmX$ of finite type over $\ok$, and a closed subscheme $\rmY$ (not necessarily smooth),
the map $\Gamma _{\rmY}(\rmX, \E) \ra \Gamma_{\rmY \times \bA^1}(\rmX \times {\mathbb A}^1, \E)$
is a weak-equivalence.
\item $\E$ satisfies Zariski excision. Recall Zariski excision means the following: for each open immersion $\rmU \ra \rmX$ of schemes
in $\Sm/\ok$ with $\rmW \subseteq \rmX$ closed (not necessarily smooth) and contained in $\rmU$, the induced map $\Gamma_{\rmW}(\rmX, \E) \ra \Gamma_{\rmW}(\rmU, \E)$ is 
a weak-equivalence.
 \item $\E$ satisfies  Nisnevich excision.  Recall that Nisnevich excision means the following: for each $\rmX' \ra X$ an \'etale map in $\Sm/\ok$, and $ \rmW \subseteq \rmX$ closed
  (not necessarily smooth),  $\rmW' = \rmW{\underset \rmX \times}\rmX'$ so that the induced map $\rmW' \ra \rmW$ is an isomorphism, the induced map $\Gamma_{\rmW}(\rmX, E) \ra \Gamma_{\rmW'}(\rmX', E)$ is a weak-equivalence.
\item $\E$ satisfies Zariski localization, i.e., if $\rmX$ is  in $\Sm/\ok$ and $\rmY \subseteq \rmX$ is a closed
subscheme (not necessarily smooth), one obtains the stable cofiber sequence: $\Gamma_{\rmY}(\rmX, \E) \ra \Gamma (\rmX, \E) \ra \Gamma (\rmX - \rmY, \E)$.
\item $\E$ has the property that the ${\mathbb P}^1$-suspension induces a weak-equivalence. i.e., If $\rmX$ is $\in \Sm/\ok$ and $\rmY$ is a closed subscheme 
(not necessarily smooth), then the  map $\Gamma_{\rmY}(\rmX, \E) \ra \Gamma_{\rmY \times {0}}(\rmX \times \bA^1, \E)$
induced by the ${\mathbb P}^1$-suspension is a weak-equivalence.
\item $\E$ has the homotopy purity property. i.e., Let $\rmZ \subseteq \rmY \subseteq \rmX$ be closed immersions of schemes in $\Sm/\ok$. Let $\cN$ denote the
normal bundle associated to the closed immersion $\rmY \subseteq \rmX$. Let $\rmB(\rmX, \rmY)$ denote the deformation space obtained by blowing up
$\rmY\times 0 $ in $\rmX \times \bA^1$. Then $\cN \ra \rmB(\rmX, \rmY)$ and $\rmX \ra \rmB(\rmX, \rmY)$ are closed immersions and induce weak-equivalences:
\[\Gamma_{\rm Z}(\cN, \E) \leftarrow \Gamma_{\rmZ \times \bA^1}(\rmB(\rmX, \rmY), \E) \ra \Gamma_{\rmZ}(\rmX, \E).\]
\item (Extension to inverse systems of smooth schemes with affine structure maps.) If $\{\rmX_{\alpha}| \alpha\}$
denotes an inverse system of smooth schemes in $\Sm/\ok$ with affine structure maps, we let $\rmX= \limalpha \rmX_{\alpha}$.
Then we define $\Gamma (\rmX, \E) = \colimalpha \Gamma(\rmX_{\alpha}, \E)$. We require that, so defined, the left-hand-side
is independent (up to weak-equivalence) of the choice of the inverse system $\{\rmX_{\alpha}|\alpha\}$ whose
inverse limit is $\rmX$.
\item (Invariance under purely inseparable maps.) If $f:\rmW' \ra \rmW$ is a purely inseparable map of smooth schemes, 
$\rmY \subseteq \rmW$ is a closed subscheme with $\rmY'=\rmY{\underset \rmW \times}\rmW'$, then the induced map 
$f^*:\Gamma_{\rmY}(\rmW, \E) \ra \Gamma_{\rmY'}(\rmW', \E)$ is a weak-equivalence provided
 the presheaf of homotopy groups of $\E$ are $\ell$-primary torsion, where $\ell \ne char(k)$. 
    \end{enumerate}
Observe that the Zariski localization and Nisnevich excision properties imply that the spectrum has 
what is often called the Brown-Gersten property and therefore cohomological  descent on the Nisnevich site. Therefore, one 
obtains the weak-equivalence of  spectra for any $\E \eps \Spt_{\rmS^1}(\ok)$: 
\be \begin{equation}
     \label{hypercoho.vs.mappingspectra}
  \H(\rmX, \E) \simeq \Map(\Sigma_{\rmS^1}\rmX, \E)
\end{equation} \ee
\vskip .3cm \noindent
where the right-hand-side is the usual mapping spectrum, and where $\H(\rmX, \E)$ denotes the hypercohomology
spectrum computed on the Nisnevich site. One may also verify that spectrum representing algebraic 
K-theory (viewed as an ${\rm S}^1$-spectrum), satisfies all of the above properties.  We show below in Lemma 
~\ref{smah.preservers.good.props} that if $\E$ is a spectrum satisfying the above properties, then  the spectrum
\[{\overset m {\overbrace {{{\E {\overset L {\underset {\mbS} \wedge}} \H(Z/{\ell})}} \cdots {{\E {\overset L {\underset {\mbS} \wedge}} \H(Z/{\ell})}}}}}\]
\vskip .3cm \noindent
also satisfies the same properties, where $\ell$ is a fixed prime different from the characteristic of $\ok$ and the right-hand-side involves 
the derived smash product of the Eilenberg-Maclane spectrum over the sphere spectrum as in \cite{C11}.
\begin{remarks} (i) Observe that the properties above, except for Nisnevich excision and the last property, together with the hypothesis that the homotopy groups of
 the spectrum $\E$ are $\ell$-primary torsion are the hypotheses needed in \cite{Yag} to ensure rigidity for the cohomology theory represented by the
spectrum $\E$.
\vskip .3cm
(ii) There are different, but related notions of spectra in the motivic context. The ${\rm S}^1$-spectra here mean spectra in the usual sense.
On replacing ${\rm S}^1$ by ${\mathbb P}^1$ one obtains  ${\mathbb P}^1$-spectra. These two categories are related by considering
the intermediate category of bi-spectra where there are two suspension functors: one, the simplicial suspension with respect to ${\rm S}^1$ and the
other, the suspension with respect to ${\mathbb G}_m$. See 
\cite{Voev3} for a discussion of these different categories. A key observation for us
is that the spectrum representing algebraic K-theory is a ${\mathbb P}^1$-spectrum and therefore one may consider the associated bi-spectra
 to apply results on ${\rm S}^1$-spectra to it.
\end{remarks}

A main result in this paper is the following rigidity statement for Borel style generalized equivariant cohomology. 
Let $\ok$ denote a fixed algebraically closed
field, $\oF$ a field containing $\ok$ and let $\bar \oF^{sep}$ ($\bar \oF$) denote a separable closure of $\oF$
(an algebraic closure of $\oF$, \res). We make the following observations that will enable us to simplify
our constructions and arguments. 
\be \begin{enumerate}[\rm(i)]
\label{field.simplifications}
     \item 
    First, if $\oF$ is of {\it infinite transcendence degree} over $\ok$, one may write $\oF$
as a direct limit of sub-fields $\oF_{\alpha}$ of finite transcendence degree over $k$. Then the algebraic closure
of $\oF$ will be the direct limit of the algebraic closures of the sub-fields $\oF_{\alpha}$. This  observation
enables us to effectively reduce to considering fields $\oF$ of finite transcendence degree over $\ok$, i.e.,
to generic points of schemes in $\Sm/k$.
\item 
Second, the algebraic closure of a field $\oK$ identifies with the splitting field of all polynomials 
in one variable over $\oK$. Therefore, it may be written as a direct limit of a sequence of finite {\it normal}
extensions $\oK_{\beta}$, each of which will be the composition of a finite separable and a finite purely 
inseparable extension. Moreover the automorphism  group of each such extension will identify with the Galois group
of the corresponding finite separable extension. Finally making use of the property: 
{\it Invariance under purely inseparable maps} considered above, together with the above observations, make it
unnecessary to distinguish between the separable closure and the algebraic closure of any field under
consideration.
\end{enumerate}
For example, $\oF$ could be the field of 
rational functions
on a $\ok$-variety and $\bar \oF$ an algebraic closure of $\oF$.
 Let $\rmG _{\oF}=\{\rmG_{\a}|\a \eps \A \}$  denote the absolute Galois group 
of the field $ \oF$. We then view $\o\oSpec \, \ok$ with the trivial action of $\rmG _{\oF}$ as a pro-scheme. 
$\o\oSpec \bar \oF = \{(\o\oSpec \, \oF_{\a})| \a \eps \A\}$
is clearly a pro-scheme together with an action of $\rmG _{\oF}$ on $\o\oSpec \, \bar \oF$, i.e., a family of compatible actions by 
$\rmG _{\a} = Gal_{\oF}(\oF_{\a})$ on $\oF_{\a}$ which is a finite Galois extension of $\oF$ with Galois group $\rmG _{\a}$. Let $\bY= \{\rmY_{\a}| \a \eps \A\}$
denote a pro-object of smooth schemes of finite type over $\ok$ provided with compatible actions by $\rmG _{\oF} = \{\rmG _{\a}| \a \eps \A\}$.
The $\rmG _{\oF}$-equivariant hypercohomology
 $\H_{\rmG _{\oF}}$ is defined in Definition ~\ref{def.Borel.coh.pro}: see ~\eqref{equiv.ge.coho.profinite.grps}.
 {\it Observe that this is always computed on the Nisnevich site}.
 \vskip .2cm
 We will adopt the following terminology. Here (as elsewhere in the paper),  $\ell$ is a prime 
different from the characteristic of $\ok$. $\E/{\ell} = \E \wedge M(\ell)$ denotes the smash product of $\E$ with a 
mod$-\ell$ Moore-spectrum, $\H(\Z/{\ell})$ denotes the Eilenberg-Maclane spectrum with
the non-trivial homotopy group in degree $0$ where it is $\Z/{\ell}$. Given a commutative ring spectrum $\rmR$ and module spectra $\rmM$, $\rmN$ over $\rmR$,
$\rmM{\overset L{\underset {\rmR} \wedge}}\rmN$ denotes the derived smash-product defined as $\rmM{\underset {\rmR} \wedge}\tilde \rmN$, where $\tilde \rmN \ra \rmN$ is
a suitable cofibrant replacement of $\rmN$. 
\vskip .2cm 
 \begin{theorem}
\label{main.thm.1}Assume that the base field $\ok$ is algebraically closed. 
 Let $\E$ denote any ring spectrum in $\Spt_{S^1}(\ok)$ satisfying the hypotheses in ~\ref{spectra.props}. (i) Then the Galois-equivariant map $\o\oSpec \, \bar \oF \ra \o\oSpec\, \ok$ 
(where the Galois group acts trivially on $\o\oSpec \, \ok$)
induces a weak-equivalence in the following cases:
\be \begin{align}
\H_{\rmG_{\oF}}(\oSpec \, \ok \times \bY, \E/\ell) &\ra \H_{\rmG_{\oF}}(\oSpec \, \bar \oF \times \bY, \E/\ell), \notag\\
\H_{\rmG_{\oF}}(\oSpec \, \ok \times \bY, \E{\overset L {\underset {\mbS} \wedge}}\H(\Z/{\ell})) &\ra 
\H_{\rmG_{\oF}}(\oSpec \, \bar \oF  \times \bY, \E{\overset L {\underset {\mbS} \wedge}}\H(\Z/{\ell})) \mbox{ and } \notag\\
\H_{\rmG_{\oF}}(\oSpec \, \ok \times \bY, \E{\overset m {\overbrace {{{{\overset L {\underset {\mbS} \wedge}} \H(Z/{\ell})}} \cdots {{{\overset L {\underset {\mbS} \wedge}} \H(Z/{\ell})}}}} })
  &\ra \H_{\rmG_{\oF}}(\oSpec \, \bar \oF  \times \bY, \E {\overset m {\overbrace {{{ {\overset L {\underset {\mbS} \wedge}} \H(Z/{\ell})}} \cdots {{ {\overset L {\underset {\mbS} \wedge}} \H(Z/{\ell})}}} }   }) \notag
\end{align} \ee
 for any positive integer $m$.
\vskip .3cm \noindent
(ii) In particular, the above weak-equivalences hold with the spectrum $\E={\mathbf K}$ which denotes the 
 spectrum  representing 
algebraic K-theory. 
\vskip .2cm \noindent
Taking $\bY = \o\oSpec \, \ok$, we obtain:
\begin{equation}
 \label{K.k}
\bK(\rmB\oG_{\oF})\compl_{\ell} \simeq  \bK(\o\oSpec \, \oF) \compl_{\ell} 
\end{equation}
\vskip .1cm \noindent
where $\compl_{\ell}$ denotes the derived completion at the prime $\ell$ as in \cite[(3.0.6)]{CJ23} and we define 
\[\bK(\rmB{\oG_{\oF}}) = \H_{\oG_{\oF}}(Spec \, k, {\bK}).\]
 \end{theorem} \ee

\vskip .3cm
For the rest of the paper, we will assume that either $\rmG_{\oF}$ is a  pro-$\ell$ group, where $\ell$ is a fixed prime different from the 
characteristic of $k$ (or alternatively that $k$ is of characteristic $0$). Let ${\rm EG}^{C}_{\oF}$ denote the
pro-scheme defined in as in ~\eqref{EGC.1} (\eqref{EGC.2}, \res).
\be \begin{theorem}
\label{main.thm.2}
 (i)   Assume $\E$ denotes a ring spectrum and the base field $\ok$ is algebraically closed as in the last theorem.
Then one obtains the weak-equivalence:
\[\H_{\rmG_{\oF}}(\chi \times {\rm EG}^C_{\oF}, \E) \simeq \Map(\chi \times_{\rmG_{\oF}} {\rm EG}^C_{\oF}, \E)  \]
where $\chi = \{\cX_{\a}| a\}$ is a pro-scheme with an action by $\rmG_{\oF}= \{\rmG_{\a}|a \}$ and with each $\cX_{\a} \eps\Sm/k$ as 
in Definition ~\ref{Hypercoho.Carl}. The left-hand-side is the
Borel-style generalized equivariant cohomology defined in ~\ref{equiv.ge.coho.profinite.grps} with respect to the spectrum $\E$, $\Map$ denotes the
obvious mapping spectrum and the term on the right-hand-side is defined as in Definition ~\ref{Hypercoho.Carl}.
\vskip .3cm
(ii) Assume $\E$ is a ring spectrum as in (i). Then we  obtain the weak-equivalence:
\[ \Map(\oSpec \, \ok \times_{\rmG_{\oF}} {\rm EG}^C_{\oF}, \E_{\ell}) \simeq \Map(\oSpec \, \bar \oF \times_{\rmG_{\oF}} {\rm EG}^C_{\oF}, \E_{\ell})\]
where $\E_{\ell}$ denotes one of the following spectra: $ \E/\ell = \E\wedge M(\ell)$, $ \E{\overset L {\underset {\mbS} \wedge}}\H(\Z/{\ell})$ or
$ \E{\overset m {\overbrace {{{{\overset L {\underset {\mbS} \wedge}} \H(Z/{\ell})}} \cdots {{{\overset L {\underset {\mbS} \wedge}} \H(Z/{\ell})}}}} }$ for some $m \ge 1$.
\vskip .3cm
(iii) Denoting $(\oSpec \, \ok) \times_{\rmG_{\oF}} {\rm EG}^C_{\oF}$ by ${\rm BG}^C_{\oF}$, one then obtains:
\[ \bK({\rm BG}^C_{\oF})\compl_{\ell} \simeq \bK(\oSpec \, \oF \times {\rm EG}^C_{\oF})\compl_{\ell} \footnote{The right-hand-side term identifies with
$\bK(\oSpec \, \oF)\compl_{\ell}$ by \cite[Corollaries, 5.3 and  5.4]{C13}. One may also observe that one of the remaining main results of \cite{C13} is the statement that
$\bK( {\rm BG}^C_{\oF})\compl_{\ell}$ identifies with the derived completion  $\bK^{\rmG_{\oF}}(\oSpec \, \ok)\compl_{\rho_{\ell}}$, where $\rho_{\ell}: \bK^{\rmG_{\oF}} (\oSpec \, \ok) {\overset {\rho} \ra} 
\bK(\oSpec \, \ok) \ra \bK(\oSpec \, \ok){\underset {\mbS} \wedge}{\mathbb H}({\mathbb Z}/\ell)$ is the obvious map.
}\]

 \end{theorem} \ee
\vskip .3cm
With the above theorems in place, we are able to continue our work in \cite{CJ26-2}
where a derived completion
theorem is proven for the actions pro-finite Galois group, resulting in a proof of the
above mentioned conjecture of the first author (see \cite[Conjecture 1.3]{C11}) in many important cases.

\section{Borel style equivariant K-theory and generalized equivariant cohomology for actions of pro-groups}
The first step in defining Borel style generalized equivariant cohomology, is to define the 
 {\it geometric classifying space of a single linear algebraic group} and the corresponding Borel construction.
 This is discussed in detail in \cite[section 2.1]{CJ23} and we refer the reader to that discussion. We will freely adopt the terminology from there
  in the following discussion and also augment that discussion with the following observation.

\begin{proposition}
 \label{good.pair.l.group}
Let $\rmG$ be a finite $\ell$-group with $\ell$ a prime different from the characteristic of $\ok$. Then there exist $\ok$-rational representations
$\rho_i$, $i=1, \cdots, s$ so that the action of $\rmG$ on $\Pi_{i=1}^s S(\rho_i)$ is free where $S(\rho_i)$ denotes the complement of
the origin in the representation $\rho_i$. 
\end{proposition}
\begin{proof}We proceed by induction on $ n = log_{\ell}(\# (\rmG))$. For $n = 1$, this result is clear since
any cyclic $\ell$-group acts freely on a single $S(\rho)$. Suppose the result is known for $n < i$, and we are
given an $\ell$-group of order $\ell^i$. The group $\rmG$ is solvable, and therefore has a quotient ${\rm Q}$ of order $\ell^{i-1}$.
Therefore, there exist representations $\rho_1, \cdots,  \rho_t$ of ${\rm Q}$ (and therefore of $\rmG$) so that the action of ${\rm Q}$
on 
\[\Pi_{j=1}^t S(\rho_j)\]
is free. The kernel ${\rm K}$ of the projection $\rmG \ra {\rm Q}$ is a cyclic group of order $\ell$, and is in the center of
$\rmG$. Let $\sigma$ denote a complex representation of ${\rm K}$ so that the ${\rm K}$-action is free on $S(\sigma)$. Construct
the induced representation $\tau = i^{\rmG}_{\rm K}(\sigma)$ of $\rmG$. It is clear from the usual induction/restriction formulae that
the restriction of $\tau$ to ${\rm K}$  also has the property that the ${\rm K}$ action on $S(\tau)$ is free. It now follows
directly that the action of $\rmG$ on
\[S(\tau ) \times \Pi_{j=1}^t S(\rho_j)\]
is free. \end{proof}
The following discussion appears in more detail in \cite[section 2.1]{CJ23}.
\subsection{Borel-style equivariant K-theory and generalized equivariant cohomology for actions of algebraic groups}
\label{subsection:geom.class.space}
 Let $\rmG$ denote a linear algebraic group over $\rmS =\oSpec \quad k$. 
Let $\{(\rmW_m, \rm\rmU_m)|m \ge 1\}$ denote an admissible gadget for $\rmG$. 
 For $m \ge   1$, we let  
\[\E\rmG^{gm,m}=\rmU_m, B\rmG^{gm,m}=\rmV_m=\rmU_m/G.\] 
We let 
\[\E\rmG^{gm} = \colimm \rmU_m \mbox{ and } B\rmG^{gm} = \colimm \rmV_m\]
 where the colimit is taken
over the  closed embeddings 
$\rmU_m \ra U_{m+1}$ and $\rmV_m \ra V_{m+1}$ corresponding to the embeddings 
$Id \times  \{0\} : \rmU_m \ra \rmU_m \times W \subseteq U_{m+1}$. These are viewed as sheaves on 
$({\rm {Sm/k}})_{Nis}$ or on $({\rm {Sm/k}})_{et}$. 
\vskip .3cm
Given a smooth scheme $\rmX$ of finite type over $S$ with a $\rmG$-action, we let $\rmU_m{\underset {\rmG} \times} X$ denote the obvious 
{\it balanced product}, 
where $(u, x)$ and $(ug^{-1}, gx)$ are identified for all $(u, x) \eps \rmU_m \times X$ and $g \eps \rmG$.  Since the $\rmG$-action on 
$\rmU_m$ is free, a 
geometric quotient again exists at least as an algebraic space in this setting.
\begin{definition} 
\label{equiv.gen.coh}
(Borel style equivariant K-theory and generalized cohomology)
 Assume first that $\rmG$ is an algebraic group or a finite group viewed as an algebraic group by imbedding it in some $\GL_n$. 
We define the Borel style equivariant K-theory of $\rmX$ to be the spectrum $\bK(\E\rmG^{gm}{\underset {\rmG} \times}X)$. More generally, given a spectrum $\E$, 
we define the corresponding Borel style generalized equivariant cohomology of $\rmX$ to be 
\[\H_{\rmG}(\rmX, \rmE)=\H(\E\rmG^{gm}{\underset {\rmG }\times}\rmX, \rmE)=
\holimm \H(\rmU_m{\underset {\rmG} \times}\rmX, \rmE) = \holimm \Map (\Sigma _{{\rmS}^1}\E{\rmG}^{gm,m}{\underset {\rmG} \times}\rmX, \rmE)\]
\vskip .3cm \noindent
\[= \Map (\colimm \Sigma _{{\rmS}^1}\E{\rmG}^{gm,m}{\underset {\rmG} \times}\rmX, \rmE) =\Map (\Sigma_{{\rmS}^1}\E\rmG^{gm}{\underset {\rmG} \times}\rmX, \rmE).\] 
\end{definition}
\vskip .3cm \noindent
Here the hypercohomology denotes hypercohomology computed on the
Nisnevich site of the scheme $\rmU^m{\underset {\rmG} \times}X$. $\Map$ denotes the simplicial mapping functor computed again on the
same site.  It is often convenient to identify this with
\be\begin{equation} 
\begin{split}
\H_{\GL_n}(\GL_n{\underset {\rmG} \times}\rmX, \rmE)=\H(E\GL_n^{gm}{\underset {\GL_n} \times}GL_n {\underset {\rmG} \times}\rmX, \rmE)=
\holimm \H(E\GL_n^{gm,m}{\underset {\GL_n} \times}\GL_n {\underset {\rmG} \times}\rmX, \rmE)\\
=\holimm \Map (\Sigma _{{\rmS}^1}E\GL_n^{gm,m}{\underset {\GL_n} \times}GL_n{\underset {\rmG} \times}\rmX, \rmE)
\end{split}
\end{equation} \ee
\vskip .3cm 
\subsection{Essential uniqueness of the geometric Borel construction}
\vskip .2cm
The following are shown in \cite[Appendix B]{CJ23} (see also \cite{MV}):
\begin{itemize}
 \item Let $\{\E\rmG^{gm,m}|m \}$ denote the ind-scheme obtained from an admissible gadget associated to the algebraic group $\rmG$.
Then if $\rmX$ is any scheme or algebraic space over $\ok$, then $\colimm \E\rmG^{gm,m}{\underset {\rmG} \times }X \cong
(\colimm \E\rmG^{gm,m}){\underset {\rmG}\times}X = \E\rmG^{gm}{\underset {\rmG}\times}X$. (This follows readily from the
observation that the $\rmG$ action on  $\E\rmG^{gm,m}$ is free and that filtered colimits commute with the
balanced product construction above.)
\item If $\{{\widetilde {\E\rmG}}^{gm,m}|m \}$ denotes  the ind-scheme obtained from another choice of 
an admissible gadget for the
algebraic group $\rmG$, one obtains an isomorphism $\E\rmG^{gm}{\underset {\rmG}\times }X \simeq {\widetilde {\E\rmG}}^{gm}
{\underset {\rmG}\times}X$ in the motivic homotopy category. 
\vskip .3cm
It follows therefore, that if $\E$ is any ${\mathbb A}^1$-local spectrum, then one obtains a weak-equivalence:
\[\Map(\E\rmG^{gm}{\underset {\rmG}\times}X, \E) = \holimm \{\Map(\E\rmG^{gm,m}{\underset {\rmG}\times}X, \E)|m\} \simeq
 \holimm \{\Map(\widetilde {\E\rmG}^{gm,m}{\underset {\rmG}\times}X, \E)|m\}\]
\[ = \Map ({\widetilde {\E\rmG}}^{gm}{\underset {\rmG}\times}\rmX, \E). \]
\end{itemize}
\vskip .3cm
The following is also proven in \cite[Lemma 2.9, p.1 39]{MV}:
if $\rmY$ is a smooth scheme of finite type over $\ok$, with a {\it free ${\rmG}$}-action and $\rmX$ is any smooth scheme 
with a $\rmG$-action, then the obvious map $\E\rmG^{gm}{\underset {\rmG} \times }(\rmX \times \rmY) \ra \rmX\times _{\rmG} \rmY$ is an
${\mathbb A}^1$-equivalence. (Here $\rmG$ is assumed to act diagonally on $\rmX \times \rmY$ so that this action is also free.)

 \subsection{Borel style equivariant K-theory and generalized equivariant cohomology for pro-group actions}
In this section we will extend the constructions of the last section to actions of pro-groups on pro-schemes.
\subsubsection{Construction of the inverse system }
 \label{compat.inverse.systems}
Let $\A$ denote a directed set and let  $\{\phi({\alpha}):\rmG_{\a} \ra \rmG_{\b}|\alpha: \b \ra \a \in \A\}$ denote an inverse system of 
algebraic groups indexed by $\A^{op}$. We will denote the order on $\A$ by $\le$ and assume that if
$\a \le \b$ and $\b \le \a$, then $\a = \b$. We will also assume that given any $\a \eps \A$, the full subcategory $\A/\a$ has only a finite number of objects
and morphisms, i.e., the set $\A^{op}$ viewed as a cofiltered category in the obvious way is {\it cofinite}.
For the most part we will
be interested in the case where this forms the inverse system of finite quotient groups of a profinite group $\rmG$. As shown in 
\cite[Corollary 1.1.18]{RZ}, we may assume that each of the structure maps $\phi(\alpha):\rmG_{\a} \ra \rmG_{\b}$ is surjective, so that
there are only finitely many non-trivial quotient groups for any $\rmG_{\a}$. Therefore, there is no loss of generality in assuming the
 hypothesis that each of the
subcategories $\A/\a$ has only a finite number of objects and morphisms, i.e., the underlying category
associated to $\A^{op}$ is cofiltered and cofinite.

\vskip .3cm
Next suppose $\bX=\{\rmX_{\a} | \a \eps \A \}$ is an inverse system of schemes, so that each $\rmX_{\a}$ is provided with an action by $\rmG_{\a}$ and
the these actions are compatible. Suppose further that structure maps of the inverse system $\{\rmX_{\a}|\a \eps \A\}$ are flat.
Then for each map $\alpha: b \ra a$ in $\A$, one obtains the following induced maps (of equivariant G-theory spectra):
\be \begin{equation}
\xymatrix{{\bG(\rmX_{\b}, \rmG_{\b})} \ar@<1ex>[r]  &  {\bG(\rmX_{\b}, \rmG_{\a}) } \ar@<1ex>[r]  & {\bG(\rmX_{\a}, \rmG_{\a})} }
\end{equation} \ee
\vskip .3cm \noindent
where the map
\[\bG(\rmX_{\b}, \rmG_{\b}) \ra \bG(\rmX_{\b}, \rmG_{\a})\]
  is induced 
 by restriction of the group action and the map
\[ \bG(\rmX_{\b}, \rmG_{\a}) \ra \bG(\rmX_{\a}, \rmG_{\a}) \]
is induced by pull-back 
along the flat maps $\rmX_{\a} \ra \rmX_{\b}$. One may readily see that this provides a direct system $\{\bG(\rmX_{\a}, \rmG_{\a})|\a \eps \A\}$.
A corresponding result holds for $K$-theory in the place of $\rmG$-theory: since K-theory is always 
contravariantly functorial, one need not assume that the structure maps of the inverse system $\{\rmX_{\a}|\a \eps \A \}$ are flat. 
\vskip .3cm
In view of the above discussion, we will assume that for each $\a \eps \A$, one is provided with a linear
algebraic group $\frG_{\a}$ (not necessarily part of any inverse system of groups indexed by $\A$) and that the inverse system of linear  algebraic groups $\{\bG_{\a}| \a \eps \A\}$ is defined
with 
\be \begin{equation}
     \label{bfG.def}
\bG_{\a} = \Pi_{b \le a}\frG_{\b}
    \end{equation} \ee
\vskip .3cm \noindent
with the structure maps of the inverse system $\{\bG_{\a}| \a \eps \A\}$ defined by the obvious projection
$\bG_{\a} =\Pi_{b \le a} \frG_{\b} \ra \bG_x = \Pi_{y \le x}\frG_y$, for $x \le a$.
\vskip .3cm
Next assume that one is given an inverse system of algebraic groups $\{\rmG_{\a}|\a \eps \A \}$, (which for the most part will be 
finite groups, viewed
as algebraic groups) $\{\rmG _{\a}| \a \eps \A\}$. Let $\phi_{a,b}: \rmG_{\a} \ra \rmG_{\b}$ denote the structure map of the inverse system
$\{\rmG_{\a}| \a \eps \A\}$. Then we choose  $\frG _{\a}$ for each $\a \eps \A$, to be a linear algebraic group containing $\rmG_{\a}$ as a closed
linear algebraic subgroup. Now one may observe that the 
 mapping $\rmG_{\a} \ra \bG_{\a}$ defined by sending $\rmG_{\a}$ by the identity map to $\rmG_{\a}$
and by the structure map $\phi_{a, b}$ to $\rmG_{\b}$, for $b \le a$, provides  a strict map $\{\rmG _{\a}| \a \eps \A\} \ra \{\bG_{\a}| \a \eps \A\}$
of inverse systems which in each degree is a closed immersion and a group homomorphism.
\vskip .3cm
Next we consider non-negative integral valued functions (or sequences) $s$ defined on $\A$.
 Given two sequences
$s$ and $t$, we let $s \le t$ if $s(a) \le t(a)$ for all $ \a \eps \A$. Clearly given two such sequences
$s$ and $t$, one has a third sequence $u$ dominating both $s$ and $t$ so that the collection of
such sequences is a directed set. We say a sequence $s$ is {\it non-decreasing} if $s(b) \le s(a)$ for
all $b \le a$ in $\A$. Since each object $\a \eps \A$ has only finitely many objects $b \le a$, we 
see that the subset of non-decreasing sequences is cofinal in the directed set of all sequences. We will denote this 
directed set of non-decreasing sequences on $\A$ by $\K$.
\vskip.3cm
\begin{definition}
 \label{class.spaces}
In this situation, we assume that for each $ a\eps \A$ one has chosen a good pair for the group $\frG_{\a}$ and that 
$\{E\frG_{\a}^{gm,m}|m \}$ is  the corresponding admissible gadget, fixed
throughout the following discussion. Associated to the sequence $s$, we now let
\be \begin{equation}
     \rmU_{\a}^s = \E\bG_{\a}^{gm,s} =\Pi_{b \le a} \E\frG^{gm, s(b)}_{\b}.
    \end{equation} \ee
\vskip .3cm \noindent
We also let
\[\E\bG^{gm,s} = \{\E\bG^{gm,s}_{\a}|a\}.\]
\end{definition}
One observes that for each fixed sequence $s$ as above, $\{\rmU_{\a}^s|\a \eps \A\}$ defines an inverse system of
schemes (with the structure map $\rmU_{\a}^s \ra \rmU_{\b}^s$ induced by the obvious projection for each $b \le a$)
and so that each $\rmU_{\a}^s$ is provided with an action by $\bG_{\a}$ which are compatible as $a$ varies. 
\begin{remark}
 For many situations, it suffices to consider non-negative integer valued sequences $s$ defined on $\A$ that are
{\it constant}. In this case we will denote such constant sequences by their common integral value. However, this will
{\it not suffice} in general. Anticipating this, we are setting up 
a framework here that will apply to more general situations readily. 
\end{remark}
\begin{definition} (Borel style generalized equivariant cohomology for pro-objects.)
\label{def.Borel.coh.pro}
Assume in addition to the above situation that $\bX=\{\rmX_{\a}| \a \eps \A\}$ is an inverse system of schemes
with compatible actions by $\{\bG_{\a}|\a \eps \}$. Let $\rmM$ denote a motivic $\rmS^1$-spectrum. Then we let
\be \begin{align}
     \label{equiv.ge.coho.pro.grps.2}
\rmH(\E\bG^{gm, s} \times_{\bG} \bX, \rmM) &= \holimt \colima \{\H(\E\bG_{\a}^{gm, s} \times_{\bG_{\a}} \rmX_{\a}, P_{\le t(a)}\rmM)|\a \eps \A\} \\
&= \holimt \colima \{\H(\Pi_{b \le a}\E{\mathfrak G}_{\b}^{gm, s(b)} \times_{\Pi_{b\le a}{\mathfrak G}_{\b}} \rmX_{\a}, P_{\le t(a)}\rmM)|\a \eps \A\} \mbox{ } \notag\\
\H_{\bG}(\bX, \rmM) &= \holims \H(\E\bG^{gm,s} \times_{\bG} \bX, \rmM) = \holims \holimt \colima \{\H(\E\bG_{\a}^{gm, s} \times_{\bG_{\a}} \rmX_{\a}, P_{\le t(a)}M)|\a \eps \A\}\notag \end{align} \ee
\vskip .3cm
In particular, when the spectrum $\E$ denotes $K$-theory, one obtains the weak-equivalences (since the $\bG_{\a}$-action on $\E\bG_{\a}^{gm, \bs(a)}$  is free):
\be \begin{align}
\bK(\E\bG^{gm, s} \times \bX, \bG) &= \colima \{\bK(\E\bG_{\a}^{gm, s(a)} \times \rmX_{\a}, \bG_{\a})|\a \eps \A\} \simeq \colima \{\bK(\E\bG_{\a}^{gm, s(a)} \times_{\bG_{\a}} \rmX_{\a})|\a \eps \A\} \mbox{ and }\notag \\
\bK(\E\bG^{gm}\times \bX, \bG) &\simeq \holims \colima\{\bK(\E\bG_{\a}^{gm, s(a)} \times_{\bG_{\a}} \rmX_{\a})|\a \eps \A\} \notag \end{align} \ee
\vskip .3cm \noindent
A corresponding result holds for G-theory, when the structure maps of the inverse system $\bX=\{\rmX_{\a}|\a \eps \A\}$ are flat. 
\vskip .3cm
Next let $\rmG=\{\rmG_{\a}| \a \eps \A\}$ denote an inverse system of finite groups, acting on the inverse system of schemes ${\bX}=\{X_{\a}|\a \eps \A\}$, 
with $\rmG_{\a}$ acting on $\rmX_{\a}$.
Then we let ${\mathfrak G}_{\a} $ denote a general linear group into which $\rmG_{\a}$ admits a closed immersion. We then form the inverse system
$\{\bG_{\a}= \Pi_{\b \le a} {\mathfrak G}_{\b}|\a \eps \A\}$ where the structure maps are the obvious projections. 
Now each $\rmG_{\a}$ admits a closed immersion into $\bG_{\a}$ by mapping into each ${\mathfrak G}_{\b}$ through the homomorphism into $\rmG_{\b}$.
We then let $\bX_{\a}= \bG_{\a} {\underset {\rmG_{\a}} \times}\rmX_{\a}$ and observe the weak-equivalence
\be \begin{equation}
\label{equiv.ge.coho.profinite.grps}
     \H_{\bG}(\bX , \E) =\holims \colima \H(\rmE\bG_{\a}^{gm, s(a)} \times_{\bG_{\a}} \bG_{\a}{\underset {\rmG_{\a}} \times}\rmX_{\a}, \E) 
\simeq \holims \colima \H(\rmE\bG_{\a}^{gm, s(a)}\times _{\rmG_{\a}}\rmX_{\a}, \E)
    \end{equation} \ee
Therefore, we denote $\holims \colima \H(\rmE\bG_{\a}^{gm, s(a)}\times _{\rmG_{\a}}\rmX_{\a}, \E)$ by $\H_{\rmG}({\mathcal X}, \E)$.
\end{definition}
\subsection{Essential uniqueness of the geometric classifying spaces for pro-group actions}
In this section, we proceed to prove that the Borel style generalized equivariant cohomology considered above is independent of
the choice of the geometric classifying spaces for the groups involved. The proof of this statement follows  in outline the proof
of the corresponding statements for actions of inverse systems of topological groups on inverse systems of topological spaces. 
Recall 
this makes use of the slice towers for the corresponding spectrum: see \cite[Proposition 11.17]{Fausk}. Our proof, therefore, will make use of the corresponding slice towers 
 as defined in \cite{Voev2} and \cite{Lev}. Next, we digress to recall the relevant properties
of the slice towers.
\vskip .3cm
Recall that $\Spt_{S^1}(k)$ denotes the category of {\it motivic $S^1$ spectra}.
We will begin with the following rather technical (but useful) result.
\begin{lemma}
 \label{smah.preservers.good.props}
Assume $\E$ is a ring spectrum in $\Spt_{\rmS^1}(k)$ satisfying the properties in ~\ref{spectra.props}. Then, for any prime $\ell \ne char(k)$,  the
spectrum 
$\E{\overset m {\overbrace {{{ {\overset L {\underset {\Sigma} \wedge}} \H(Z/{\ell})}} \cdots {{ {\overset L {\underset {\Sigma} \wedge}} \H(Z/{\ell})}}}}}$
in $\Spt_{\rmS^1}(k)$ also satisfies the same properties.
\end{lemma}
\begin{proof} Clearly it suffices to consider the case $m=1$. The hypothesis that a spectrum $\E$ has Nisnevich excision, implies it has what is
called the Brown-Gersten property (see \cite[section 3, Definition 1.13]{MV}), so that it has Nisnevich descent. Now one may observe readily that if $\E$ has Nisnevich
excision and Zariski localization, then so does $\E{\overset L {\underset {\Sigma} \wedge}} \H(Z/{\ell})$. Therefore, for every scheme $\rmX \in \Sm/\ok$ and a closed
subscheme $\rmY \subseteq \rmX$, the obvious map $\Gamma_{\rmY}(\rmX, \E{\overset L {\underset {\Sigma} \wedge}} \H(Z/{\ell})) \ra
R\Gamma_{\rmY}(\rmX, E{\overset L {\underset {\Sigma} \wedge}} \H(Z/{\ell}))$ is a weak-equivalence. This readily proves that all the properties
in ~\ref{spectra.props}(i) through (vii) hold for $\E{\overset L {\underset {\Sigma} \wedge}} \H(Z/{\ell})$ if they hold for $\E$.
\end{proof}
\vskip .3cm

  The work of \cite{Lev} and \cite{Pel} show that now one can define a sequence of functors $f_n: \Spt_{S^1}(k) \ra \Spt_{S^1}(k)$
so that the following properties are true:
\subsubsection{}
\begin{enumerate}[\rm(i)]
\label{props.mot.Post}
 \item One obtains a map $f_n \E \ra \E$ that is natural in the spectrum $\E$. 
\item One also obtains a tower of maps $\cdots \ra f_{n+1}\E \ra f_n \E \ra \cdots  \ra f_0\E = \E$. Let $s_p\E$ denote the
canonical homotopy cofiber of the map $f_{p+1}\E \ra f_p\E$ and let $\rmP_{\le q-1}\E$ \footnote{We choose to denote this functor by $\rmP_{\le q-1}$ and not by the more common notation
$s_{\le q-1}$ to prevent any possible confusion with the sequences denoted $s$ appearing in Definition ~\ref{class.spaces}.} be defined as the canonical homotopy cofiber of the map $f_q\E \ra \E$. 
Then one may show readily (see, for example, \cite[Proposition 3.1.19]{Pel}) that the canonical homotopy fiber of the induced map $\rmP_{\le q}\E \ra \rmP_{\le q-1}\E$ identifies also with $s_q(\E)$. One then also obtains a tower of fibrations $\cdots \ra \rmP_{\le q}\E \ra \rmP_{\le q-1}\E \ra \cdots$. 
\item Let $\rmY $ be a smooth scheme of finite type over $\ok$ and let $\rmW \subseteq \rmY$ denote a closed not necessarily smooth subscheme so that
$codim_{\rmY}(\rmW) \ge q$ for some $q \ge 0$. Then the map $f_q\E \ra \E$ induces a weak-equivalence (see 
\cite[Lemma 7.3.2]{Lev}:
\[\Map(\Sigma_{S^1}(\rmY/\rmY- \rmW)_+, f_q\E) \ra \Map(\Sigma_{S^1}(\rmY/\rmY- \rmW)_+, \E), \E \eps \Spt_{S^1}(\ok).\]
\item It follows that, then, 
\[\Map(\Sigma_{S^1}(\rmY/\rmY- \rmW)_+, {\rm P}_{\le q-1}\E) \simeq *\]
under the same hypotheses on $\rmY$ and $\rmW$. (This makes use of the property that if $\E$ is homotopy invariant and
has Nisnevich excision and Zariski localization, then the terms $f_p\E$ and hence $s_p\E$ have the same properties. These follow readily
by identifying the terms $f_p\E$ with the terms in the homotopy coniveau tower as in \cite{Lev}.)
\item Let $\rmX$ be a smooth scheme of finite type over $\ok$ and of dimension $d$, and let $n$ be a non-negative integer. 
 Then, since $\E$ assumed to be $-1$-connected, $\pi_n(\Map (\Sigma_{S^1}\rmX_+, s_p\E))=0$, if $p >dim_k(\rmX)+n$: see \cite[proof of Proposition 2.1.3]{Lev}. 
\item Recall that $s_{p+m}\E $ identifies with the homotopy fiber of the
induced map $\rmP_{\le p+m}\E \ra \rmP_{\le p+m-1}\E$.  It follows by making use of the Milnor exact sequence
for 
$\pi_n(\holimm \Map(\Sigma_{S^1}\rmX_+, {\rmP}_{\le p+m}\E)) $, that for $p >dim_k(\rmX)+n+1$, 
\[\pi_n(\holimm \Map(\Sigma_{S^1}\rmX_+, {\rmP}_{\le p+m}\E))  \cong \pi_n(\Map(\Sigma_{S^1}\rmX_+, {\rmP}_{\le p}\E)).\]
To see this, consider the long-exact-sequence for $i=0, 1$, with $m \ge 1$:
\be \begin{multline}
     \begin{split}
\pi_{n+i}\Map(\Sigma_{S^1}\rmX_+, s_{p+m}\E) \ra \pi_{n+i}\Map(\Sigma_{S^1}\rmX_+, {\rmP}_{\le p+m}\E) \ra \pi_{n+i}\Map(\Sigma_{S^1}\rmX_+, \rmP_{\le p+m-1}\E) \\ \notag
\ra \pi_{n+i-1}(\Map(\Sigma_{S^1}\rmX_+, s_{p+m}\E)) \notag \end{split} \end{multline}
If $p> dim_k(\rmX)+n+1$, then $p+m > dim_k(\rmX)+n+1+m \ge dim_k(\rmX)+n+1 > dim_k(\rmX)+n > dim_k(\rmX)+n-1$, so that
\[\pi_{n+j}(\Map(\Sigma_{S^1}\rmX_+, s_{p+m}\E)) =0, \mbox{ for } j=-1, 0, 1\]
so that the map $\rmP_{\le p+m}\E \ra \rmP_{\le p+m-1}\E$ induces an isomorphism
\[\pi_{n+i}(\Map(\Sigma_{S^1}\rmX_+, {\rmP}_{\le p+m}\E)) \cong \pi_{n+i}(\Map(\Sigma_{S^1}\rmX_+, {\rmP}_{\le p+m-1}\E)), i=0,1.\]
\item
Now one may make use of \cite[Proposition 2.1.3]{Lev} to conclude that, if $\E$ is also homotopy invariant, then 
$\pi_n(\holimm \Map(\Sigma_{S^1}\rmX_+, {\rmP}_{\le p+m}\E)) \cong \pi_n(\Map(\Sigma_{S^1}\rmX_+, \E)) $ so that if $p >dim_k(\rmX)+n+1$, then
\[\pi_n(\Map(\Sigma_{S^1}\rmX_+, {\rmP}_{\le p}\E)) \cong \pi_n(\Map(\Sigma_{S^1}\rmX_+, \E)).\]
\end{enumerate}
\begin{remarks} 
 It is shown in \cite{Pel} that the functors $f_q$, $s_q$ and
$\rmP_{\le q}$ lift to the level of model categories, though this fact is not important for us.
\end{remarks}
\vskip .3cm
Next we proceed to adapt the slice towers in the context of an inverse system of Borel constructions as
in Definition ~\ref{def.Borel.coh.pro}. We will consider associated to each $t \eps \K$,
\be \begin{align}
     \label{Post.equiv.ge.coho.pro.grps}
\H(\E\bG^{gm, s} \times_{\bG} \bX, \rmP_{\le t}\E) &= \colima \{\H(\E\bG_{\a}^{gm, s} \times_{\bG_{\a}} \rmX_{\a}, {\rmP}_{\le t(a)}\E)|\a \eps \A\} \\
\H_{\bG}(\bX, \rmP_{\le t}E) &= \holims \H(\E\bG^{gm,s} \times_{\bG} \bX, \rmP_{\le t}\E) \notag \\
\holimt \H_{\bG}(\bX, \rmP_{\le t}\E) &\simeq \holims \holimt \H(\E\bG^{gm,s} \times_{\bG} \bX, \rmP_{\le t}\E) \notag \end{align} \ee
\vskip .3cm

Next assume that $\rmG$ is a linear algebraic group and $(\rmW, \rmU)$,  $(\bar \rmW,\bar \rmU)$ are both good pairs  for $\rmG$. Let $\{(\rmW_m, \rm\rmU_m)|m \ge 1\}$
 and $\{(\bar \rmW_m, \bar \rm\rmU_m)|m \ge 1\}$ denote the associated admissible gadgets. Then, since $\rmG$ acts freely on both $\rm\rmU_m$ and $\bar \rm\rmU_m$, it 
is easy to see  that $(\rmW \times \bar \rmW, \rmU \times \bar \rmW \cup \rmW \times \bar \rmU)$ is also a good pair for $\rmG$ with
respect to the diagonal action on $\rmW \times \bar \rmW$. Moreover, under the same hypotheses,  one may readily see that $\{\rmW_m \times \bar \rmW_m, \rm\rmU_m \times \bar \rmW_m \cup \rmW_m \times \bar \rm \rmU_m)|m \ge 1\}$
 is also an admissible gadget for $\rmG$ for the diagonal action on $\rmW \times \bar \rmW$: see for example, \cite[(B.0.1)(iv)]{CJ23}. Let $\rmX$ denote a smooth scheme of finite type over $\ok$
on which $\rmG$ acts.
\vskip .3cm
Since $\rmG$ acts freely on both $\rm\rmU_m$ and $\bar \rm\rmU_m$, it follows that $\rmG$ has a free action on $\rm\rmU_m \times \bar W_m$ and also on
$\rmW_m \times \bar \rm\rmU_m$. We will let $\widetilde \rm\rmU_m = \rm\rmU_m \times \bar \rmW_m \cup \rmW_m \times \bar \rm\rmU_m$ for the following discussion.
One may now compute the codimensions 
\be \begin{align}
\label{codims}
codim_{\widetilde \rm\rmU_m{\underset {\rmG} \times}\rmX}({\widetilde \rmU}_m {\underset {\rmG} \times}\rmX- \rm\rmU_m \times \bar \rmW_m{\underset {\rmG} \times}\rmX) &= codim_{\rmW_m}(\rmW_m - \rm\rmU_m), \mbox{ and }  \\
codim_{\widetilde \rm\rmU_m{\underset {\rmG} \times}\rmX}({\widetilde \rmU}_m{\underset {\rmG} \times}\rmX - \rmW_m \times \bar \rm\rmU_m{\underset {\rmG} \times}\rmX) &= codim_{\bar \rmW_m}(\bar \rmW_m -\bar \rm\rmU_m) 
\end{align} \ee
\vskip .3cm \noindent
In view of ~\ref{props.mot.Post}(iv), it follows that the induced maps
\be \begin{multline}
     \begin{split}
\label{uniqueness.pro.1}
Map(\Sigma_{S^1}\rm\rmU_m \times _{\rmG}\rmX_+, P_{\le q-1}\E) \simeq Map(\Sigma_{S^1}(\rm\rmU_m \times \bar \rmW_m) \times_{\rmG}\rmX_+, P_{\le q-1}\E)\\
 \ra Map(\Sigma_{S^1}{\widetilde \rmU}_m\times _{\rmG} \rmX_+, P_{\le q-1}\E) \mbox{ and }\\
 Map(\Sigma_{S^1}\bar \rm\rmU_m \times _{\rmG}\rmX_+, P_{\le q-1}\E) \simeq Map(\Sigma_{S^1}(\bar \rmW_m \times \rm\rmU_m ) \times_{\rmG}\rmX_+, P_{\le q-1}\E) \\
\ra Map(\Sigma_{S^1}{\widetilde \rmU}_m\times _{\rmG} \rmX_+, P_{\le q-1}\E)
\end{split} \end{multline} \ee
are both weak-equivalences if $codim_{\rmW_m}(\rmW_m - \rm\rmU_m)$ and $codim_{\bar \rmW_m}(\bar \rmW_m -\bar \rm\rmU_m) $ are both 
greater than or equal to $q$. The first weak-equivalences in ~\eqref{uniqueness.pro.1} are provided by the homotopy property for the 
spectrum $\E$, which is inherited by the slices (as shown in \cite[(2.2)(2)]{Lev}.)
\vskip .3cm
Next assume that $\rmY$ is a smooth scheme of finite type over $\ok$ provided with a {\it free} action by $\rmG$. Now $\rmG$ acts freely on 
both $\rmW_m \times \rmY \times \rmX$ as well as on $\rm\rmU_m \times \rmY \times \rmX$. One may compute
\be \begin{equation}
     \label{codim.Y}
codim_{\rmW_m \times _{\rmG} (\rmX \times \rmY)}(\rmW_m \times _{\rmG} (\rmX \times \rmY) - \rm\rmU_m \times_{\rmG}(\rmX \times \rmY)) = codim_{\rmW_m}(\rmW_m - \rm\rmU_m)
    \end{equation} \ee
\vskip .3cm \noindent
Again, in view of ~\ref{props.mot.Post}(iv), it follows that the induced map
\be \begin{multline}
\label{uniqueness.pro.2}
     \begin{split}
Map(\Sigma_{S^1}\rm\rmU_m \times _{\rmG}(\rmX \times \rmY)_+, P_{\le q-1}\E) \ra Map(\Sigma_{S^1}\rmW_m  \times_{\rmG}(\rmX\times \rmY)_+, P_{\le q-1} \E)\\
 \simeq Map(\Sigma_{S^1}(\rmX\times _{\rmG} \rmY)_+, P_{\le q-1}\E) 
\end{split} \end{multline} \ee
 is a weak-equivalence if $codim_{\rmW_m}(\rmW_m - \rm\rmU_m) \ge q$. The last weak-equivalence is again provided by the homotopy property for the 
spectrum $\E$, which is inherited by the slices (as shown in \cite[(2.2)(2)]{Lev}). One may recall that if the admissible gadgets are chosen as in 
\cite[(2.1.1)]{CJ23}, then  $codim_{\rmW_m}(\rmW_m- \rm\rmU_m) = m.codim_{\rmW}(\rmW - \rmU)$
and $codim_{\bar \rmW_m}(\bar \rmW_m-\bar \rm\rmU_m) = m.codim_{\bar \rmW}(\bar \rmW- \bar \rmU)$, so that both $codim_{\rmW_m}(\rmW_m - \rm\rmU_m) \ge q$ and
$codim_{\bar \rmW_m}(\bar \rmW_m -\bar \rm\rmU_m) \ge q$ if $m$ is chosen large enough.
\vskip .3cm
Next assume the situation considered in Definition  ~\ref{def.Borel.coh.pro}. Assume in addition that for each $b \eps \A$,
$(\rmW_{\b}, \rmU_{\b})$ and $(\bar \rmW_{\b}, \bar \rmU_{\b})$ are two choices of good pairs and that $\{(\rmW_{b,m}, \rmU_{b,m}|m \ge 1\}$ and 
$\{(\bar \rmW_{b,m}, \bar \rmU_{b,m})| m \ge 1\}$ are the corresponding associated admissible gadgets chosen as in \cite[(2.1.1)]{CJ23}. We will denote the
geometric classifying spaces obtained from $\{ (\rmW_{b,m}, \rmU_{b, m})|m \ge 1\}$ ($\{(\bar \rmW_{b,m}, \bar \rmU_{b,m})|m \ge 1\}$ by
$\{\E {\rmG}_{\b}^{gm}|b, m\}$ ($\{\bar \E {\rmG}_{\b}^{gm}|b , m\}$, \res). Assume further that
one is provided with compatible actions by $\bG_{\a} = \Pi_{b \le a} \frG_{\b}$ on $\rmX_{\a}$ and $\rmY_{\a}$ which are smooth schemes
of finite type over $\ok$ and that the action of $\bG_{\a}$ and hence of $\rmG_{\a}$ on $\rmY_{\a}$ is free. 
\begin{theorem}
 \label{uniqueness.pro.grp.actions}
Assume the above situation. (i) Then one obtains weak-equivalences:
\be \begin{equation}
      \holimt \holims \colima \H(\E\bG^{gm,s}_{\a}\times _{\rmG_{\a}}\rmX_{\a}, P_{\le t}\E) \simeq \holimt \holims \colima \H(\bar \E{{\bG}}_{\a}^{gm,s}\times _{{ {\rmG_{\a}}}}\rmX_{\a}, P_{\le t}\E).
\end{equation} \ee
\vskip .3cm \noindent
(ii) $ \holims \colima \H(\E\bG_{\a}^{gm,s}\times _{\rmG_{\a}}(\rmX_{\a} \times \rmY_{\a}), \E) \simeq \holimt \holims \colima \H(\E\bG_{\a}^{gm,s}\times _{\rmG_{\a}}(\rmX_{\a} \times \rmY_{\a}), P_{\le t}\E)$ 
\vskip .3cm
$ \simeq  \holimt \colima \H(\rmX_{\a}\times _{\rmG_{\a}} \rmY_{\a}, P_{\le t}\E) \simeq \colima \H(\rmX_{\a}\times _{\rmG_{\a}} \rmY_{\a}, \E). $
\end{theorem}
\begin{proof} For each fixed sequence $ t \in \K$, one may clearly choose a sequence $ s \eps \K$ so that for each $ a \in \A$,
 $s(a)codim_{\rmW}(\rmW - \rmU) \ge t(a)+1$ and $s(a)codim_{\bar \rmW}(\bar \rmW- \bar \rmU) \ge t(a)+1$. It now follows from ~\eqref{uniqueness.pro.1} that with this choice of $s$, one obtains a
weak-equivalence:
\[\colima  \H(\E\bG_{\a}^{gm,s}\times _{\rmG_{\a}}\rmX_{\a}, P_{\le t(a)}\E) \simeq \colima \H({\bar E}{{\bG}}_{\a}^{gm,s}\times _{{ {\rmG_{\a}}}}\rmX_{\a}, P_{\le t(a)}\E)\]
It follows that for each fixed $t \eps \K$, one obtains a weak-equivalence:
\[\holims \colima  \H(\E\bG_{\a}^{gm,s}\times _{\rmG_{\a}}\rmX_{\a}, P_{\le t(a)}\E) \simeq \holims \colima \H({\bar E}{ {\bG}}_{\a}^{gm,s}\times _{{ {\rmG_{\a}}}}\rmX_{\a}, P_{\le t(a)}\E)\]
Taking $\holimt$ now, we obtain the first weak-equivalence in (i). Observe that the two homotopy inverse limits over $s$ and $t \eps \K$
commute. Therefore, one may now take the outer homotopy inverse limit to be $ \holims$. 
 This completes the proof of (i).
\vskip .3cm
Next we consider (ii). Recall $\E\bG_{\a}^{gm,s} = \Pi_{b \le a} \E\frG_{\b}^{gm, s(b)} = \Pi_{b \le a} \rmU_{\b}^{s(b)}$ and that $\bG_{\a} = \Pi_{b \le a}\frG_{\b}$. 
Therefore, the weak-equivalence in (ii) is
provided by ~\eqref{uniqueness.pro.2} which shows for each fixed $t \eps \K$, one may choose an $s \eps \K$ so that the map
\[\colima \H(\Pi_{b \le a}\rmU_{\b}^{s(b)}\times _{\rmG_{\a}} (\rmX_{\a} \times \rmY_{\a}), P_{\le t(a)}\E) 
 \ra  \colima \H(\Pi_{b \le a}\rmW_{\b}^{s(b)}\times _{\rmG_{\a}} (\rmX_{\a} \times \rmY_{\a}), P_{\le t(a)}\E) \]
is a weak-equivalence. It follows therefore that the map
\[\holimt \holims \colima \H(\Pi_{b \le a}\rmU_{\b}^{s(b)}\times _{\rmG_{\a}} (\rmX_{\a} \times \rmY_{\a}), P_{\le t(a)}\E) 
 \ra  \holimt \holims \colima \H(\Pi_{b \le a}W_{\b}^{s(b)}\times _{\rmG_{\a}} (\rmX_{\a} \times \rmY_{\a}), P_{\le t(a)}\E) \]
is also a weak-equivalence. 
As in ~\eqref{uniqueness.pro.2}, the homotopy property for the spectrum $\E$, which is inherited by the slices
then provides the weak-equivalence:
\[\holimt \holims \colima \H(\Pi_{b \le a}\rmW_{\b}^{s(b)}\times _{\rmG_{\a}} (\rmX_{\a} \times \rmY_{\a}), P_{\le t(a)}\E) \simeq 
 \holimt \colima \H(\rmX_{\a}\times_{\rmG_{\a}} \rmY_{\a}, P_{\le t(a)}\E).
\]
Since $\rmX_{\a}$ and $\rmY_{\a}$ are schemes of finite type over $\ok$, for each non-negative integer $n$, one may choose a sequence $t \eps \K$, 
so that $t(a) \ge dim_k(\rmX_{\a} \times _{\rmG_{\a}} \bY_{\a}) +n+2$. Then, ~\ref{props.mot.Post}(vi) shows that the induced map
\[\colima \H(\rmX_{\a}\times _{\rmG_{\a}} \rmY_{\a}, P_{\le t(a)}\E) \ra \colima \H(\rmX_{\a}\times _{\rmG_{\a}}  \rmY_{\a}, \E) \]
is an isomorphism on the $n$-th homotopy groups. Therefore, the last weak-equivalence in (ii) follows.

\end{proof}
\begin{corollary}
 Assume $\E$ denotes any one of the following spectra (a) the ${\mathbb P}^1$-ring spectrum, $\bK$ representing algebraic K-theory on $\Sm/k$ or
 (b) the ${\mathbb P}^1$-ring spectrum 
${\overset m {\overbrace {{{\bK {\overset L {\underset {\mbS} \wedge}} \H(Z/{\ell})}} \cdots {{\bK {\overset L {\underset {\mbS} \wedge}} \H(Z/{\ell})}}}}}$
for some $m \ge 1$.
 Assume also that the remaining hypotheses of the last theorem hold. Then we obtain the
weak-equivalences:
\[\holims \colima \H(\E\bG_{\a}^{gm,s}\times _{\rmG_{\a}}\rmX_{\a}, \E) \simeq \holims \colima \H({\bar E}{ {\bG}}_{\a}^{gm,s}\times _{{ {\rmG_{\a}}}}\rmX_{\a}, \E) \mbox{ and}\]
\[\holims \colima \H(\E\bG_{\a}^{gm,s}\times _{\rmG_{\a}}(\rmX_{\a} \times \rmY_{\a}), E) \simeq \colima \H(\rmX_{\a}\times _{\rmG_{\a}} \rmY_{\a}, \E).\]
\end{corollary}
\begin{proof} The ${\rm S}^1$-spectrum representing algebraic K-theory corresponding to the ${\mathbb P}^1$-spectrum 
satisfies all the required properties listed in ~\ref{spectra.props}. Moreover, it satisfies the defining property to be an
 $\Omega_{{\rm S}^1}$-spectrum except for the
$0$-th term. Therefore, the results of the last theorem apply.
\end{proof}

\section{Rigidity for mod-$\ell$ Borel style generalized cohomology and equivariant K-theory} 
This section is devoted to a complete proof of Theorem ~\ref{main.thm.1}. The main result is Theorem ~\ref{main.thm1.1} whose proof
depends on certain technical results that are discussed later on in this section, in Propositions ~\ref{ext.rigidity} and ~\ref{Phi.preserves.fibrations}.
The strategy of the proof of Theorem ~\ref{main.thm.1} is as follows. 
We will {\it first assume $\rmG$ is a finite group acting on  smooth schemes $\rmX$ and $\rmY$ of finite type over $\ok$}.
We will assume that $\rmG$ is provided with a closed immersion into the finite product $\bG_1 \times \cdots \times \bG_k$ of algebraic groups 
and that $\bm=(m_1, \cdots, m_k)$ is a  chosen sequence of  non-negative integers. First we fix some notation: for $\bm=(m_1, \cdots, m_k)$, we let
\be \begin{align}
     \rmU^{\bm} &= \E\bG_1^{gm, m_1} \times \cdots \times \E\bG_k^{gm,m_k} \times \rmY,\\
      \rmV^{\bm} &= \E\bG_1^{gm, m_1} \times \cdots \times \E\bG_k^{gm,m_k}{\underset {\rmG}\times }\rmY, \notag \\
    \rmU^m(i) &= \E\bG_i^{gm, m} \times Y \mbox{ and } \rmV^m(i) = \E\bG_i^{gm, m} {\underset {\rmG}\times} \rmY\notag
    \end{align}
\vskip .3cm \noindent
where the geometric classifying spaces are defined making use of some choice of admissible gadgets as in \cite[4.2]{MV} and $\times$ denotes the fibered product over the base scheme $\oSpec \, \ok$.
Let $\pi= \pi^{\bm}: \rmU^{\bm} {\underset {\rmG}\times}X = \E\bG_1^{gm, m_1} \times \cdots \times \E\bG_k^{gm,m_k} {\underset {\rmG}\times} ( Y \times X) \ra \rmV^{\bm}$  denote the corresponding projection. 
We begin with the weak-equivalences (where the hypercohomology is computed on the Nisnevich site):
\be \begin{align}
     \label{Leray.1}
\H(\rmU^{\bm}{\underset {\rmG}\times}\rmX, \E) &\simeq \H(\rmV^{\bm}, R\pi^{\bm  }_*E).
 \end{align} \ee
\vskip .3cm \noindent
This follows readily from the definition of $\rmR\pi^{\bm}_*(\E)$ as $\pi^{\bm }_*({\mathcal G}\E)$ where $\E \ra {\mathcal G}\E$ is a fibrant
replacement or from the definition of $\rmR\pi^{\bm }_*$ if $\E$ is already assumed to be fibrant in the injective model structure. 
One may also observe that
\be \begin{align}
\H(\E\bG_1^{gm} \times \cdots \times \E\bG_k^{gm}{\underset {\rmG}\times}(\rmY \times \rmX), \E) &= \holimbm \H(\rmU^{\bm}{\underset {\rmG}\times}\rmX, \E) \mbox{ and }\notag \\
\H(\E\bG_1^{gm} \times \cdots \times \E\bG_k^{gm}{\underset {\rmG}\times}\rmY, R\pi^{\bm }_*(\E)) &= \holimbm \H(\rmV^{\bm}, R\pi^{\bm }_* (\E)).\notag
\end{align} \ee
\vskip .3cm
Next we proceed to analyze the stalks of the presheaf $\rmR\pi^{\bm}_{ *}(\E)$. Afterwards we will extend the framework to include in the place of the 
scheme $\rmX$, also pro-objects
in the category of smooth schemes of finite type over $\ok$ provided with actions by inverse systems of groups. We will then show that the required rigidity statements follow by taking 
$\rmX$ to be either $\oSpec \, \ok$ or $\oSpec \, \bar \oF$. 

\vskip .3cm
Observe first that the points of $\colimbm \rmV^{\bm}$ are the residue fields at points of $\cup_{\bn \ge 0}(
S^{n_1, \cdots, n_k})$ where $S^{n_1, \cdots, n_k}= ( \rmU^{ n_1}(1)-  \rmU^{n_1-1}(1)) \times \cdots \times 
( \rmU^{n_k}(k)-  \rmU^{ n_k-1}(k)) {\underset {\rmG}\times} Y$. Let $x_{\bn}$ denote such a
fixed point belonging to the stratum $S^{\bn}$, for some $\bn =\ndots \ge 1$. Observe that if $\O_{\bm}^{x_{\bn}, h}$ denotes the Henselization of $\O_{\rmV^{\bm}}$ at the
point $x_{\bn}$, then there are compatible maps $ \O_{\bm}^{x_{{\bn}, h}} \ra \O_{{\mathbf m}'}^{x_{\bn}, h} $
for ${\mathbf m}' \ge \bm$.
\vskip .3cm
Next we proceed to compute the stalks of the Nisnevich sheaf $\rmR\pi^{\bm}_*(\E)$ at a point $x_{\bn} \eps S^{\bn}$. 
One obtains the following identification, in view of Lemma ~\ref{cartesian} below, 
\vskip .3cm
$(\rmR\pi^{\bm}_*(\E))_{x_{\bn}} = \colimV \H(\rmV{\underset {\rmV^{\bm}} \times}U^{\bm}{\underset {\rmG}\times}X, \E) \simeq \H(\oSpec \, (\O_{\rmV^{\bm}, x_{\bn}}^{\it h}){\underset {\rmV^{\bm} }\times}U^{\bm} {\underset {\rmG}\times}X, \E)$.
\vskip .3cm
The induced map $\oSpec \, (\O_{\rmV^{\bm}, x_{\bn}}^h){\underset {\rmV^{\bm}} \times}\rmU^{\bm} \ra \oSpec \, (\O_{\rmV^{\bm}, x_{\bn}}^{\it h})$ is a finite \'etale map since $\rmG$ is a finite group and
 the action of $\rmG$ on $\rmU^{\bm}$ is free. Therefore,  $\oSpec \, (\O_{\rmV^{\bm}, x_{\bn}}^h){\underset {\rmV_{\bm}} \times}\rmU^{\bm }\ra \oSpec \, (\O_{\rmV^{\bm}, x_{\bn}}^{\it h})$ breaks up as the
 finite disjoint union of Hensel local rings $\rmB_i$: see \cite[Theorem 4.2 and Corollary 4.3, Chapter I]{Mi}. Since $\rmG$ acts freely on $\rmU^{\bm}$, one can see that each
fiber over the point $x_{\bn}$ of $\pi^{\bm}$ consists of the closed points of these local rings which are permuted by $\rmG$. 
Therefore, since each $\oSpec \, \rmB_i$ must be connected, one may in
fact observe that elements of $\rmG$ act transitively on $\sqcup _i \oSpec \, \rmB_i$ and that all the Hensel rings
 $\rmB_i$ are isomorphic. (To see $\oSpec \, \rmB_i$ must be connected, one may argue as follows. Suppose 
 that it is not connected. Then there exist nonzero idempotents
$e_i$, $i=1, 2$, so that $e_1+e_2 =1$ and with $e_1. e_2=0$.  But since $\rmB_i$ is a local ring, it has no idempotents
other than $0$ and $1$.)
It follows one obtains the isomorphism 
\be \begin{equation}
     \label{stalk.decomp.0}
\oSpec \, (\O_{\rmV^{\bm}, x_{\bn}}^h){\underset {\rmV^{\bm}} \times}\rmU^{\bm} \cong \sqcup _{g \eps G} \oSpec \, \rmB
\end{equation} \ee

\vskip .3cm \noindent
where $\rmB$ denotes any one of the Hensel rings $\rmB_i$. Therefore, we obtain the identification
\[\rmV{\underset {\rmV^{\bm}} \times}U^{\bm}{\underset {\rmG}\times}X \cong (\oSpec \, \rmB) {\underset {\oSpec \, \ok} \times} \rmX,\]
and hence the identification of the stalks:
\be \begin{equation}
\label{stalks.ident.0}
     (\rmR\pi^{\bm}_*(\E))_{x_{\bn}} \simeq \H( \oSpec \, \rmB {\underset {\oSpec \, \ok} \times} \rmX, \E)
    \end{equation} \ee
\vskip .3cm \noindent
{\it Moreover, the definition of $\oSpec \, \rmB$ shows that it is independent of $\rmX$}.
\vskip .3cm
Next suppose $\{\cX_{\a} | \a \eps \A\}$, $\{\cY_{\a}| \a \eps \A\}$  are inverse systems of schemes and $\{\rmG_{\a} |\a \eps \A\}$ is an inverse system of 
finite groups, so that each 
$\cX_{\a}$ and $\cY_{\a}$ is provided with an action by the  group $\rmG_{\a}$ and
that these actions are compatible. Suppose further that the structure maps of the inverse systems 
$\{\cX_{\a}|\a \eps \A\}$ and $\{\cY_{\a}| \a \eps \A\}$ are 
flat as maps of schemes. In this situation, we may start with a closed immersion $\rmG_{\a} \ra {\mathfrak G}_{\a}$ for each
$\a \eps \A$, with ${\mathfrak G}_{\a}$ being a linear algebraic group. Then we may form an inverse system of 
algebraic groups by replacing each ${\frG}_{\a}$ with the the finite product $\bG_{\a}=\Pi_{b \le a}{\mathfrak G}_{\b}$ where the
structure maps are induced by the obvious projection maps: see the construction in ~\ref{compat.inverse.systems}
for more details. Let $\{s(a)|\a \eps \A\}$ denote a sequence with each $s(a)$ a non-negative integer. Now
we let
\be \begin{align}
\label{Us(a)}
      \rmU_{\a}^{{\mathbf s}(a)} &= \Pi_{b \le a} \E{\mathfrak G}_{\b}^{s(b)} \times \cY_{\a}, \\
     \rmV_{\a}^{{\mathbf s}(a)} &= (\Pi_{b \le a} \E{\mathfrak G}_{\b}^{s(b)}){\underset {\rmG_{\a}} \times} \cY_{\a} \notag
    \end{align} \ee
\vskip .3cm \noindent
We let $\pi^{{\bs}(a), \cX_{\a}}:    \rmU_{\a}^{{\mathbf s}(a)}{\underset {G_{\a}} \times} \cX_{\a} \ra \rmV_{\a}^{\bs(a)}$ 
 denote the corresponding projections. 
\vskip .3cm
A point $x$ of $\{\rmV_{\a}^{\bs(a)}|\a \eps \A\}$ corresponds to a compatible collection of points $\{x_{\a}| \a \eps \A\}$ with each $x_{\a}$ being 
a point of
$\rmV_{\a}^{\bs (a)}$. Then the analysis above (see ~\eqref{stalk.decomp.0}) shows that one obtains a decomposition:
\be \begin{equation}
\label{stalk.decomp.1}
\oSpec \, (\O_{\rmV_{\a}^{\bs (a)}, x_{\a}}^h){\underset {\rmV_{\a}^{\bs(a)}} \times} \rmU_{\a}^{\bs(a)} \cong \sqcup _{{g_{\a}} \eps G_{\a}} \oSpec \, B_{\a}
\end{equation} \ee
\vskip .3cm \noindent
Let $\E_{\a}$ denote the restriction of the given presheaf of spectra $\E$ to the Nisnevich site of $ \rmU_{\a}^{\bs(a)}{\underset {G_{\a}} \times} \cX_{\a}$. In view of 
Proposition ~\ref{Phi.preserves.fibrations}, we obtain the identification:
\vskip .3cm
$\rmR\pi^{\bs, \cX}_{*}(\Phi(\{\E_{\a}|\a \eps \A\}))= \Phi(\{\rmR\pi^{\bs(a), \cX}_{a*}\E_{\a})|\a \eps \A\})$
\vskip .3cm \noindent
where $\pi_{\a}^{\bs(a), \cX}: Top( \rmU^{\bs(a)}_{\a}{\underset {G_{\a}} \times} \cX_{\a}) \ra Top(\rmV^{\bs(a)})$ and 
$\pi^{\bs, \cX}: Top(\{ \rmU^{\bs(a)}_{\a}{\underset {G_{\a}} \times}\cX_{\a}| \a \eps \A\}) \ra Top(\{\rmV_{\a}^{\bs(a)}| \a \eps \A\})$ 
are
the induced maps of sites. Therefore, (see Corollary ~\ref{stalk.ident}) one obtains the following identification of the stalk at 
$x = \{x_{\a} |\a \eps \A\}$:
\be \begin{equation}
     \label{stalks.ident.1}
\rmR\pi^{\bs,\cX}_{*, x}\Phi(\{\E_{\a}| \a \eps \A\}) \simeq \colima \H(\oSpec \rmB_{\a} {\underset {\oSpec \, k} \times} \cX_{\a}, \E_{\a})
\end{equation} \ee

\vskip .3cm \noindent
where $\oSpec \, \rmB_{\a}$ depends only on $ \rmU_{\a}^{\bs(a)}$ and not on $\cX_{\a}$. Next observe that the diagonal 
imbedding of $\A$ in $\A \times \A$ is cofinal, so that
denoting by $\B$ another copy of $\A$, one may identify the stalks above with
\be \begin{equation}
     \label{stalk.ident.2}
\begin{split}
 \colimA \colimB \H(\oSpec \, \rmB_{\a} {\underset {\oSpec \, k} \times} \cX_{\b}, \E) \simeq \colimA \colimB \Map_{\Sigma_{\rmS^1}}(\Sigma_{\rmS^1}\cX_{b+} \wedge (\oSpec \, \rmB_{\a})_+, \E)\\
\simeq \colimA \colimB \Map_{\Sigma_{\rmS^1}}(\Sigma_{\P^1}\cX_{b+}, \Map_{\Sigma_{\rmS^1}}(\Sigma_{\rmS^1}(\oSpec \, \rmB_{\a})_+, \E))
\end{split}
\end{equation} \ee
\vskip .3cm \noindent
where $\Map_{\Sigma_{\rmS^1}}$ denotes the internal hom in the category of $\P^1$-spectra on the Nisnevich site of $\Sm/k$.
\vskip .3cm
Next assume that the field $\ok$ is algebraically closed and that $\oF$ is a field containing $\ok$. In view of our remarks in ~\ref{field.simplifications},
 we may assume without loss of generality that $\oF$ has finite transcendence degree over $k$. Let $\rmG= \rmG_{\oF}$ denote the
absolute Galois group of $\oF$. Let $\bX=\{\cX_{\a}| \a \eps \A\}$ denote the constant inverse system given by $\cX_{\a} = \oSpec \, \ok$, 
where $\ok$ is given algebraically closed field $\ok$ 
and provided with the trivial action by the absolute Galois group ${\rm G}_{\oF}=\{\rmG_{\a}| \a \eps \A\}$ of the field $\oF$.
Let $\bZ= \{\cZ_{\a}| \a \eps \A\}$ denote the inverse system of finite normal extensions $\{\oSpec \, \oF_{\a} | \a \eps \A\}$ 
of the field $\oF$
 together with the obvious action of ${\rm G}_{\oF}$ on $\oSpec \, \bar \oF$, i.e., a family of compatible actions by 
$\rmG_{\a} = Gal_{\oF}(\oF_{\a})$ on $\oF_{\a}$. (Observe that a finite normal extension is the composition of a finite separable
extension and finite purely inseparable extension. The automorphism group of such a finite normal extension identifies with
the automorphism group of the corresponding separable extension, i.e., with the Galois group $\rmG_{\a}$.)
 In this situation, the obvious map of pro-objects $\bZ=\{\cZ_{\a}|\a \eps \A\} \ra  \bX=\{\cX_{\a} | \a \eps \A\}$ induced by the 
imbeddings $\{\ok \ra \oF \ra \oF_{\a}| \a \eps \A\}$
is ${\rm G}_{\oF}$-equivariant. Therefore, one obtains an inverse system of commutative diagrams for each fixed 
sequence $\bs$:
\vskip .3cm
\be \begin{equation}
     \label{compat.maps}
\xymatrix{{ \rmU_{\a}^{\bs(a)}{\underset {G_{\a}} \times}\cZ_{\a}} \ar@<1ex>[r] \ar@<1ex>[d]^{\pi^{\bs(a), \cZ}_{\a}} & { \rmU_{\a}^{\bs(a)}{\underset {G_{\a}} \times}\cX_{\a}} \ar@<1ex>[d]^{\pi^{\bs(a), \cX}_{\a}}\\
{\rmV_{\a}^{\bs(a)}} \ar@<1ex>[r]^{id} & {\rmV_{\a}^{\bs(a)}}}
  \end{equation} \ee
\begin{lemma}
 \label{cartesian}
Let $\rmG _{\a}$ denote a finite group acting on a  smooth scheme $\cX_{\a}$ of finite type over $\ok$. Let $\pi: \rmU_{\a}^{\bs(a)}
 \ra \rmV_{\a}^{\bs(a)}$ denote the
principal $\rmG _{\a}$-bundle as before. Then for any $\rmV \ra \rmV_{\a}^{\bs(a)}$ the square
\vskip .3cm
\xymatrix{{(\rmV{\underset {\rmV_{\a}^{\bs(a)}} \times} \rmU_{\a}^{\bs(a)}) {\underset {\rmG _{\a}} \times}\cX_{\a}} \ar@<1ex>[r] \ar@<-1ex>[d] & {\rmV} \ar@<-1ex>[d]\\
{ \rmU_{\a}^{\bs(a)}{\underset {\rmG _{\a}} \times}\cX_{\a}} \ar@<1ex>[r] & {\rmV_{\a}^{\bs(a)}}}
\vskip .3cm \noindent
is a cartesian square.
\end{lemma}
\begin{proof} Let $\rmZ$ denote a scheme with compatible maps $\alpha: \rmZ \ra \rmV$, $\beta:\rmZ \ra  \rmU_{\a}^{\bs(a)}{\underset {\rmG _{\a}} \times}\cX_{\a}$
 and $\gamma: \rmZ \ra \rmV_{\a}^{\bs(a)}$. Since the $\rmG _{\a}$-action on $ \rmU_{\a}^{\bs(a)}$ is free one may  view 
$ \rmU_{\a}^{\bs(a)}{\underset {\rmG _{\a}} \times}\cX_{\a} $ as the 
quotient stack $[( \rmU_{\a}^{\bs(a)} \times \cX_{\a})/{\rmG _{\a}}]$. Therefore,
one observes that the given map $\beta : \rmZ \ra [( \rmU_{\a}^{\bs(a)} \times \cX_{\a})/{\rmG _{\a}}]$ corresponds to giving a 
principal $\rmG$-bundle
$z:\tilde \rmZ \ra \rmZ$ together with a $\rmG _{\a}$-equivariant map $\phi: \tilde \rmZ \ra  \rmU_{\a}^{\bs(a)} \times \cX_{\a}$. Now 
the map $z$
together with the given maps of $\rmZ$ provide compatible $\rmG_{\a}$-equivariant maps $\tilde \alpha = 
\alpha \circ z: \tilde \rmZ \ra \rmV$,
$\tilde \gamma = \gamma \circ z: \tilde \rmZ \ra \rmV_{\a}^{\bs(a)}$ and $\tilde \beta = \phi: \tilde \rmZ \ra 
 \rmU_{\a}^{\bs(a)} \times \cX_{\a}$.
Since the square
\vskip .3cm
\xymatrix{{\rmV{\underset {\rmV_{\a}^{\bs(a)}} \times} \rmU_{\a}^{\bs(a)} \times \cX_{\a}} \ar@<1ex>[r] \ar@<-1ex>[d] & {\rmV} \ar@<-1ex>[d]\\
{ \rmU_{\a}^{\bs(a)}\times \cX_{\a}} \ar@<1ex>[r] & {\rmV_{\a}^{\bs(a)}}}
\vskip .3cm \noindent
is evidently a cartesian square of $\rmG _{\a}$-schemes, one obtains an induced $\rmG _{\a}$-equivariant map 
$\tilde \phi:\tilde \rmZ \ra {\rmV{\underset {\rmV_{\a}^{\bs(a)}} \times} \rmU_{\a}^{\bs(a)} \times \cX_{\a}}$. Since 
${\rmV{\underset {\rmV_{\a}^{\bs(a)}} \times} \rmU_{\a}^{\bs(a)} \times \cX_{\a}} \ra {\rmV{\underset {\rmV_{\a}^{\bs(a)}} \times}
 \rmU_{\a}^{\bs(a)} {\underset {\rmG _{\a}}\times}\cX_{\a}}$ is a 
principal $\rmG _{\a}$-bundle, the map $\tilde \phi$ descends to define a map 
$\rmZ \ra {\rmV{\underset {\rmV_{\a}^{\bs(a)}} \times} \rmU_{\a}^{\bs(a)} {\underset {\rmG _{\a}} \times}\cX_{\a}} $, thereby proving the lemma.
\end{proof}

\vskip .3cm
\begin{theorem}
\label{main.thm1.1}
Assume in addition to the above situation that the spectrum $\E$ is one of the spectra as in Theorem ~\ref{main.thm.1}, i.e.,
$\E=\E'/{\ell}$, ${\E'}{\overset L {\underset {\mbS} \wedge}}\H(Z/{\ell})$ or ${\overset m {\overbrace {{{{\E'} {\overset L {\underset {\mbS} \wedge}} \H(Z/{\ell})}} \cdots {{{\E'} {\overset L {\underset {\mbS} \wedge}} \H(Z/{\ell})}}
}}}$ for some ring spectrum $\E' \eps \Spt_{\rmS^1}(k)$. 
 Then the corresponding induced maps of the stalks
\vskip .3cm
$\rmR\pi^{\bs, \cZ}_{*, x}\Phi(\{\E_{\a}|\a \eps \A\}) \ra \rmR\pi^{\bs, \cX}_{*, x}\Phi(\{\E_{\a}|\a \eps \A\})$
\vskip .3cm \noindent
is a weak-equivalence. It follows therefore, that the induced map
\vskip .3cm
$\colimA \H( \rmU_{\a}^{\bs(a)}{\underset {G_{\a}} \times} (\cY_{\a} \times \oSpec \, \ok), \E) \ra \colimA \H( \rmU_{\a}^{\bs(a)}{\underset {G_{\a}} \times} (\cY_{\a} \times \oSpec \, F_{\a}), \E)$
\vskip .3cm \noindent
is a weak-equivalence for each fixed sequence  $\bs$.  Therefore, one also obtains a weak-equivalence
on taking the homotopy inverse limit over all $\bs$ of the above colimits. i.e., The induced map 
\be \begin{multline}
     \label{rigidity.key}
     \begin{split}
\H_{{\rm G}_{\oF}}( \bY \times \oSpec \, \ok,  \E)  = \holims \colima \H( \rmU_{\a}^{\bs(a)}{\underset {G_{\a}} \times}(\cY_{\a} \times \oSpec \, k), \E) \\ 
\ra \holims \colima \H( \rmU_{\a}^{\bs(a)}{\underset {G_{\a}} \times}(\cY_{\a} \times \oSpec \, F_{\a}), \E) =\H_{{\rm G}_{\oF}}(  \bY \times \oSpec \, \bar F, \E)
\end{split}
\end{multline} \ee
\vskip .3cm \noindent
 is a weak-equivalence, where $\bY = \{\cY_{\a}| \a \eps \A\}$. In particular, taking $\bY= \oSpec \, \ok$, one obtains the weak-equivalence:
$$\H_{{\rm G}_{\oF}}(\oSpec \, \ok ,  \E)   \ra \H_{{\rm G}_{\oF}}(\oSpec \, \bar \oF, \E).$$
\end{theorem}
\begin{proof}
 The first statement of the theorem, providing the weak-equivalence of the stalks follows readily from the identification in
~\eqref{stalk.ident.2}, the observation that the choice of $\oSpec \, \rmB_{\a}$ depends only on the $ \rmU^{\bs}_{\a}$ and Proposition ~\ref{ext.rigidity} applied with the scheme
$W$ denoting $\oSpec \, \rmB_{\a}$ for a fixed $a$. This will first show that the induced maps of the term in ~\eqref{stalk.ident.2} will be a weak-equivalence on taking the colimit
over $\B$ for each fixed $\a \eps \B$. Then one may also take the colimit over $\A$ to obtain a weak-equivalence of the corresponding stalks.
\vskip .3cm
The second statement then follows readily from Proposition ~\ref{Phi.preserves.fibrations}. The third statement then follows
from the second by taking the homotopy inverse limit $\holims$. Finally the last statement follows from this
by observing from Definition ~\ref{def.Borel.coh.pro} that now $ \rmU_{\a}^{\bs(a)} = \Pi_{b \le a}\E\frG_{\b}^{s(b)}$ and that 
the product $\times $ denotes the fibered product over the base scheme $\oSpec \, \ok$, so that taking $\bY =\{\cY_{\a}|\a\} = \{\oSpec \, \ok\}$, one obtains 
\vskip .3cm
$\H_{{\rm G}_{\oF}}(\oSpec \, \ok, \E) =  \holims \colimA \H( \rmU_{\a}^{\bs(a)}{\underset {G_{\a}} \times} (\oSpec \, k), \E)$ \ and that
\vskip .3cm
$\H_{{\rm G}_{\oF}}(\oSpec \, \bar \oF, \E)= \holims  \colimA \H( \rmU_{\a}^{\bs(a)}{\underset {G_{\a}} \times} (\oSpec \, F_{\a}), \E)$.
\end{proof}
\vskip .3cm \noindent
{\bf Proof of Theorem ~\ref{main.thm.1}.}
Observe that the last theorem readily proves the first statement in Theorem ~\ref{main.thm.1} when the spectrum $\E= \E'/\ell=\E \wedge M(\ell)$.
We proceed to consider the other situations considered in Theorem ~\ref{main.thm.1}. When the spectrum $\E$ has been replaced by 
$\E{\overset L {\underset {\Sigma} \wedge}}\H(\Z/{\ell})$ the first statement of the Theorem follows as in the last case since the homotopy groups of this spectrum are all 
$\ell$-primary torsion. It follows that the same conclusion holds when the spectrum $\E$ has been replaced by the iterated derived smash product:
$\E{\overset L {\underset {\mbS} \wedge}}\H(\Z/{\ell}) {\overset L {\underset {\mbS} \wedge}}\H(\Z/{\ell}) \cdots {\overset L {\underset {\Sigma} \wedge}}\H(\Z/{\ell}) $.
 This completes the proof of the first statement of the Theorem ~\ref{main.thm.1}.
 \vskip .2cm
Next we proceed to discuss the proof of the statement ~\eqref{K.k} which is the main content of the second statement in Theorem ~\ref{main.thm.1}.
Taking ${\mathbf Y} = \oSpec \, \ok$ and $\E = \bK$ in the last statement in (i) now provides the weak-equivalence:
\be \begin{equation}
\label{main.thm.2.1}
\H_{\rmG_{\oF}}(\oSpec \, \ok, \bK{\overset m {\overbrace {{{{\overset L {\underset {\mbS} \wedge}} \H(Z/{\ell})}} \cdots {{{\overset L {\underset {\mbS} \wedge}} \H(Z/{\ell})}}}} }) \ra \H_{\rmG_{\oF}}(\oSpec \, \bar \oF, \bK {\overset m {\overbrace {{{ {\overset L {\underset {\mbS} \wedge}} \H(Z/{\ell})}} \cdots {{ {\overset L {\underset {\mbS} \wedge}} \H(Z/{\ell})}}} }   })
\end{equation}
for each fixed $m$. Observe that 
\[\H_{\rmG_{\oF}}(\oSpec \, \bar \oF, \bK {\overset m {\overbrace {{{ {\overset L {\underset {\mbS} \wedge}} \H(Z/{\ell})}} \cdots {{ {\overset L {\underset {\mbS} \wedge}} \H(Z/{\ell})}}} }   }) \simeq \holims  \colimA \H( \rmU_{\a}^{\bs(a)}{\underset {G_{\a}} \times} (\oSpec \, F_{\a}), \bK{\overset m {\overbrace {{{{\overset L {\underset {\mbS} \wedge}} \H(Z/{\ell})}} \cdots {{{\overset L {\underset {\mbS} \wedge}} \H(Z/{\ell})}}}} })\] 
Since $\oF_a$ is a finite Galois extension of $\oF$ with Galois group $\rmG_a$, $\rmG_a$ acts freely on $\oSpec \oF_{a}$, so that the 
term on the right identifies with 
\[\holims  \colimA \H( \rmU_{\a}^{\bs(a)}\times (\oSpec \, F), \bK{\overset m {\overbrace {{{{\overset L {\underset {\mbS} \wedge}} \H(Z/{\ell})}} \cdots {{{\overset L {\underset {\mbS} \wedge}} \H(Z/{\ell})}}}} } \simeq \H(\oSpec \, \oF, \bK {\overset m {\overbrace {{{{\overset L {\underset {\mbS} \wedge}} \H(Z/{\ell})}} \cdots {{{\overset L {\underset {\mbS} \wedge}} \H(Z/{\ell})}}}} }).\]
The term on the right handside denotes the hypercohomology spectrum computed on the Nisnevich site of $\oSpec \oF$. Clearly this
identifies with $\H(\oSpec \, \oF, \bK) {\overset m {\overbrace {{{{\overset L {\underset {\mbS} \wedge}} \H(Z/{\ell})}} \cdots {{{\overset L {\underset {\mbS} \wedge}} \H(Z/{\ell})}}}} } \simeq \bK(\oSpec \, \oF) {\overset m {\overbrace {{{{\overset L {\underset {\mbS} \wedge}} \H(Z/{\ell})}} \cdots {{{\overset L {\underset {\mbS} \wedge}} \H(Z/{\ell})}}}} } $.
Taking the homotopy inverse limit as $m \ra \infty$, this yields $\bK(\oSpec \, \oF)\compl_{\ell}$. Now using the definition of $\bK({\rm BG_{\oF}})\compl_{\ell}$
as the homotopy inverse limit of the term on the left-hand-side of ~\eqref{main.thm.2.1}, shows that
\[\bK({\rm BG_{\oF}})\compl_{\ell} \simeq \bK(\oSpec \, \oF)\compl_{\ell}\]
thereby completing the proof of the second statement in Theorem ~\ref{main.thm.1}.
\vskip .3cm

\subsection{Rigidity for spectra }
\begin{proposition}
\label{ext.rigidity}
(i) Let $\E$ denote a spectrum and let $\rmX \mapsto \h(\rmX, \E)= \Map(\Sigma _{{\rmS}^1}\rmX_+, \E)$ denote the generalized cohomology
theory defined with respect to the spectrum $\E$. If $\rmW$ is a fixed smooth scheme of finite type over $\ok$, then the functor $\rmX \mapsto 
\h(\rmX \times  \rmW, \E) = \Map(\Sigma _{{\rmS}^1}\rmX_+ \wedge \rmW_+, \E)$ is represented by the ${\rmS}^1$-spectrum $\Map(\Sigma _{{\rmS}^1}\rmW_+, \E)$.
This spectrum also satisfies all the properties in ~\ref{spectra.props} except possibly for Nisnevich excision and the last two properties.
In addition, if $\ell. h^*(\rmX, \E) =0$ for all smooth schemes $\rmX$ of finite type over $\ok$, then $\ell. h^*(\rmX \times \rmW, \E) =0$ also. (Here
$h^*$ denotes the homotopy groups of the spectrum $\h(\quad, \E)$  re-indexed in the obvious manner.)
\vskip .3cm
(ii) Suppose $\E$ is a spectrum so that $\ell.\h^*(\rmX, \E) =0$ for all smooth schemes $\rmX$ of finite type over $\ok$. Then if $W$ is any smooth
scheme of finite type over $\ok$, the obvious map $id \times i:\rmW {\underset {\oSpec \, k} \times} \oSpec \, \bar \oF \ra \rmW {\underset {\oSpec \, \ok} \times} \oSpec \, \ok$
(induced by the given map $i:\oSpec \, \bar \oF \ra \oSpec \, \ok$) induces an isomorphism
\vskip .3cm
$\h^*(\rmW{\underset {\oSpec \, \ok} \times} \oSpec \, \ok, \E) \ra \h^*(\rmW {\underset {\oSpec \, \ok} \times} \oSpec \, \bar \oF, \E)$
\end{proposition}
\begin{proof} (i) is more or less obvious.  However, we will provide some details for the convenience of the reader. Since $\rmX \mapsto \h(\rmX \times \rmW, \E)$ is also a cohomology theory, it suffices to verify that it extends
 to define a cohomology theory for pairs of smooth schemes $(\rmX, \rmY)$ with $\rmY$ locally closed in $\rmX$ satisfying the properties discussed
in \cite[Definition 1.2]{Yag}. One may extend such a generalized cohomology theory to such pairs $(\rmX, \rmY)$
by defining $\h_{\rmY \times \rmW}(\rmX \times \rmW, \E)$ to be the homotopy fiber of $\h(\rmX \times \rmW, \E) \ra \h(\rmY \times \rmW, \E)$.
If $\rmX$ and $\rmY$ are both smooth schemes of finite type over $\ok$ and $\rmY$ is locally closed in $\rmX$, we call $(\rmX, \rmY)$ 
a smooth pair.
Then one needs to check this extension to smooth pairs  satisfies the following:
\begin{itemize}
 \item {Suspension}: if $(\rmX, \rmY)$ is a smooth pair, then $\h_{\rmY\times \rmW\times 0}(\rmX \times \rmW \times \bA^1, \E) \simeq 
\h_{\rmY \times \rmW}(\rmX \times \rmW, \E)$
\item{Zariski excision}: if $\rmZ \subseteq \rmX_0 \subseteq \rmX$ are smooth schemes with $\rmZ$ closed in $\rmX_0$ and $\rmX_0$ open in $\rmX$,
then, one obtains a weak-equivalence: $\h_{\rmZ\times \rmW}(\rmX_0 \times \rmW, \E) \simeq \h_{\rmZ \times \rmW}(\rmX \times \rmW, \E)$
\item{Homotopy invariance}: If $(\rmX, \rmY)$ is a smooth pair, then $\h_{\rmY \times \rmW \times \bA^1}(\rmX \times \rmW \times \bA^1) \simeq
\h_{\rmY \times \rmW}(\rmX \times \rmW)$.
\item{Homotopy Purity}: If $\rmZ \subseteq \rmY \subseteq \rmX$ are closed immersions of smooth schemes, $\N$ is the 
normal bundle to the immersion of $\rmY$ in $\rmX$ and $\rmB(\rmX, \rmY)$ denotes the deformation space obtained by 
blowing up $\rmX\times \bA^1$ along $\rmY\times 0$ and removing the divisor ${\rm P}(\N)$. Then one obtains weak-equivalences:
$\h_{\rmZ\times \rmW}(\N \times \rmW) \simeq \h_{\rmZ \times \rmW \times \bA^1}(\rmB(\rmX, \rmY) \times \rmW) \simeq \h_{\rmZ \times \rmW}(\rmX \times \rmW)$.
\end{itemize}
\vskip .3cm
Of these, the first three properties follow readily from the fact they hold for the generalized cohomology theory represented by the spectrum $\E$.
 The last property may be verified
using the observation that $\rmB(\rmX, \rmY) \times \rmW \cong \rmB(\rmX \times \rmW, \rmY \times \rmW)$ and the homotopy purity theorem
proven in \cite[Theorem 2.23, p. 115]{MV}. 
\vskip .3cm
Now the statement (ii) follows from \cite[Theorem 1.10]{Yag} since the statement in (i) verifies that all the hypotheses in \cite[Theorem 1.10]{Yag}
are satisfied.
\end{proof}
\subsection{Sites for pro-schemes}
Let $\Sch (S)$ denote the category of schemes of finite type over the fixed Noetherian base scheme $S$. Observe that each scheme 
$\rmX \eps \Sch (S)$
 is automatically quasi-compact. We will let $\A$ denote a partially ordered directed set which is cofinite as in 
~\ref{compat.inverse.systems}, i.e.,  
for each fixed $\a \eps \A$, the set of elements $b \eps \A$ so that $b \le a$ is finite and the relation $\le$ is anti-symmetric. 
In particular, this means 
that there is at most one map between any two objects in $\A$.
Then a pro-object in
$\Sch (S)$ will mean an inverse system of schemes $\{\rmX_{\a}|\a \eps \A\}$. Observe that, there is at most one structure map between any two 
$\rmX_{\a}$ and $\rmX_{\b}$ that are part of the inverse system.
\vskip .3cm
A level representation of a map $f: \X \ra \Y$ of pro-schemes will mean a choice of isomorphisms $\X \cong \{\rmX_{\a}|\a \eps \A\}$ and
$\Y \cong \{\rmY_{\a}|\a \eps \A\}$ of pro-schemes together with a compatible collection of maps $\{f_{\a}:\rmX_{\a} \ra \rmY_{\a}|\a \eps \A\}$ which
represents $f$. Recall (see \cite[Appendix 3.2]{AM}) any finite diagram of maps of pro-schemes with no loops may be replaced up to isomorphism by
a diagram of level-maps.  Therefore, when considering such finite diagrams of
pro-objects, we will always consider level maps.
\vskip .3cm
One begins with a small site on $\Sch (S)$, which we will denote by $\Top$. 
The objects of the corresponding small site for a pro-object $\X = \{\rmX_{\a}| \a \eps \A\}$
will be pro-objects $\U = \{ \rmU_{\a}| \a \eps \A\}$ provided with a level-map $\U \ra \X$ of pro-objects, so that for each $\a \eps \A$, the map 
$ \rmU_{\a} \ra \rmX_{\a}$
belongs to the given site $\Top(\rmX_{\a})$.  Morphisms between two objects $\U =\{ \rmU_{\a}|\a \eps \A\} \ra \V= \{\rmV_{\a} | \a \eps \A\}$ will be obvious commutative
triangles of level maps of pro-objects over $\X$. A covering of a pro-object $\X$ will be a collection $\{\U^{\alpha}|\alpha \}$ of 
objects in the
site over $\X$ so that for each fixed $\a \eps \A$, $\{ \rmU^{\alpha}_{\a}| \alpha\}$ will be a cover of $\rmX_{\a}$.
One may contrast this with $\{\Top(\rmX_{\a})|\a \eps \A\}$ which is a fibered site, fibered over $\A$. 
\vskip .3cm

 It is clear that if  $f: \X= \{\rmX_{\a}|\a \eps \A\} \ra 
\Y = \{\rmY_{\a}|\a \eps \A\}$ is a level-map, it  induces a morphism of sites:
$f_*: \Top(\X ) \ra \Top(\Y)$ where $f_*(\U) = \{ \rmU_{\a}{\underset {\rmY_{\a}} \times}\rmX_{\a} | \a \eps \A\}$, where $\U = \{ \rmU_{\a}|\a \eps \A\}$.
\vskip .3cm
Let $a_0 \eps I$ denote fixed index and let $()_{a_0}$ denote the functor sending an object $\X=\{\rmX_{\a} | \a \eps \A\}$ to $\rmX_{a_0}$. 
Moreover, the restriction functor $()_{a_0}:\Top (\X) \ra \Top (\rmX_{a_0})$ defines a map of sites $\Top (\rmX_{a_0}) \ra
 \Top (\X)$ sending 
$\U =\{ \rmU_{\a}|\a \eps \A\} \mapsto  \rmU_{a_0}$. 
Henceforth a presheaf (sheaf) will generically denote a presheaf (sheaf) of pointed sets (pointed simplicial sets, abelian groups
or rings etc.)
Next observe that a presheaf (sheaf) on the fibered site $\{\Top(\rmX_{\a})|\a \eps \A\}$ is given by a collection of
 presheaves (sheaves)
$\{\rmF_{\a} \mbox{ on } \Top (\rmX_{\a})|\a \eps \A\}$  together with compatibility data for the structure maps of the inverse 
system. 
Let $\Psh (\X)$ ($\Sh (\X)$) denote the corresponding category of presheaves (sheaves, \res) of pointed sets on $\Top (\X)$. 
We let $\{\Psh(\rmX_{\a})|\a \eps \A\}$ denote the category of presheaves of pointed sets on the fibered site $\{\Top(\rmX_{\a})|\a \eps \A\}$.
Then we define
a functor
\vskip .3cm
$\Phi: \{\Psh (\rmX_{\a})|\a \eps \A\} \ra \Psh (\X)$
\vskip .3cm \noindent
as follows. Let $\rmF = \{\rmF_{\a}| \a \eps \A \} \eps \{\Psh(\rmX_{\a})|\a \eps \A\}$. Then
\be \begin{equation}
   \Gamma (\U, \Phi(\rmF)) = \colima \Gamma ( \rmU_{\a}, \rmF_{\a}) = \colima \Gamma (\U, ()_{a*}(\rmF_{\a})) 
\end{equation} \ee
One may now observe that $\Phi(F) = \colimi ()_{a*}(\rmF_{\a})$, where $\rmF= \{\rmF_{\a}|\a \eps \A\}$.
\vskip .3cm
\begin{definition} 
\label{points}
A point for the site $\Top(\X)$ will denote a compatible collection of points $\p=\{p_{\a}|\a \eps \A\}$ with $p_{\a}$ being
 a point of the site $\Top(\rmX_{\a})$. (Recall this means if $\lambda_{a, b}: \rmX_{\a} \ra \rmX_{\b}$ is the structure map of $\X$, then $\lambda _{a, b}\circ \p_{\a} = \p_{\b}$,
for every $a$ and $b$.)
\end{definition}
\begin{proposition}
 \label{nbds} Let $\p =\{p_{\a}|\a \eps \A\}$ denote a point of the site $\Top(\X)$ and let 
 let $ \rmU_{a_0}$ denote a neighborhood of $p_{a_0}$ in 
$\Top(\rmX_{a_0})$, for a fixed $a_0 \eps \A$. Then there exists a neighborhood $\U$ of $\p$ in $\Top(\X)$ so that the map $\U_{a_0} \ra \rmX_{a_0}$ factors
 through the given map $\rmU_{a_0} \ra \rmX_{a_0}$.
\end{proposition}
\begin{proof} We construct $\U$ in steps.
 \vskip .3cm
{\it Step 1}. Recall that $a_0$ has only finitely many descendants. A string $b_n \le b_{n-1} \le \cdots \le b_1 \le b_0=a_0$
of descendants of $a_0$ will be called a {\it path} of length $n$ starting at $a_0$ and ending at $b_n$. For each such path, one starts 
with a neighborhood of $p_{b_n}$ and takes iterated pull-backs to define neighborhoods of $p_{b_i}$, $i=n, n-1, \cdots, 1, 0$.
For each vertex $b$ on such a path, take the fibered product over $\rmX_{\b}$ of all such neighborhoods of $p_{\b}$  obtained from varying
paths starting at $a_0$ and passing through descendants of $b$. 
Finally take the fibered product over $\rmX_{a_0}$ of such a neighborhoods of $p_{a_0}$ obtained this way with 
 the given neighborhood $ \rmU_{a_0}$. At this point, we will replace the original neighborhood $ \rmU_{a_0}$ by this neighborhood: clearly
we have now a compatible system of neighborhoods for $p_{a_0}$ and all the $\{p_{\b}| b \le a_0\}$. 
\vskip .3cm
{\it Step 2}. Let $a' \eps \A$, $a' \ne a_0$. If $a'$ is a descendant of $a_0$, we have already constructed a neighborhood of $p_{a'}$
compatible with the neighborhood $ \rmU_{a_0}$. Therefore, without loss of generality, we may assume $a'$ is {\it not} a descendant of $a_0$.
Since $a'$ has only finitely many descendants, there are only finitely many paths that start at $a'$ and passing through its
descendants. 
\vskip .3cm
For each path that starts at $a'$ and ends at a descendant $d_{a_0}$ of $a_0$, one may take iterated pull-back of the 
already chosen neighborhood $ \rmU_{d_{a_0}}$ to define neighborhoods of every $p_{\b}$, for every vertex $b$ on such a path. 
For each path that starts at $a'$ and ends at a vertex $v$ that is {\it not} a descendant of $a_0$, one chooses $\rmX_{\b}$
as the neighborhood of each $p_{\b}$, for a vertex $b$ on this path. 
\vskip .3cm
For each of the descendants $b$ of $a'$ (including $a'$) there will be only finitely many such paths that start at $a'$ and pass through $b$: 
now take the iterated fibered product over $\rmX_{\b}$ of all the neighborhoods of $p_{\b}$ obtained this way by varying the 
above paths. This produces a neighborhood $ \rmU_{a'}$ of $p_{a'}$ which is compatible with  neighborhoods of $p_{\b}$ for all the
descendants $b$ of $a'$, some of which may be also descendants of $a_0$. Observe that this construction keeps intact all the
already chosen neighborhoods of $p_{\b}$, for $b$ which are descendants of $a_0$, chosen in step 1. Moreover if $a'$ is a descendant of
$a''$, then the system of neighborhoods constructed this way for paths starting at  $a''$ leaves intact the neighborhoods constructed 
for any $p_c$, where
$c$ is a descendant of $a'$. Therefore, the system of neighborhoods constructed this way for each $p_{\a}$, defines a compatible 
system of neighborhoods $\U$ for the point $\p$. Since the structure map $\U_{a_0} \ra \rmX_{a_0}$ factors through the given map 
$ \rmU_{a_0} \ra \rmX_{a_0}$ one started out with,
the last statement in the proposition is clear.
\end{proof}
\begin{corollary}
 \label{stalk.ident}
 Let $\rmF = \{\rmF_{\a}| \a \eps \A \} \eps \{\Psh(\rmX_{\a})|\a \eps \A\}$ and let $\p =\{p_{\a}|\a \eps \A\}$ denote a compatible collection of
points of $\X = \{\rmX_{\a}|\a \eps \A\}$. Then 
\be \begin{equation}
   \Phi(\rmF)_{\p} = \colima (\rmF_{\a})_{p_{\a}} 
\end{equation} \ee
\end{corollary}
\begin{proof} This is clear in view of the last statement in the last Proposition. \end{proof}
\begin{remark}
 In several cases of interest the directed set $\A$ would be a product of $\NN$, where $\NN$ is the set of natural numbers. 
In this case therefore the inverse system of schemes $\{\rmX_{\a}|\a \eps \A\}$ will reduce to a product of towers of schemes, where
each tower is indexed by $\NN$. In this case the construction in Proposition ~\ref{nbds} is particularly simple, as there is
a single path that starts at a given vertex and passes through its descendants.
\end{remark}
 
\vskip .3cm
\subsubsection{Model structures on simplicial presheaves}
The {\it local projective model structure} (as in \cite{Bl}) or the local flasque model structure (as in \cite{Isak}) will be the ones we
put on the topos of simplicial presheaves on the above site. We proceed to recall their definitions in a general context.
\vskip .3cm
Let $\C$ denote any essentially small site and let $\SPrsh (\C)$ denote the category of all pointed simplicial presheaves on $\C$. Then {\it the local projective
model structure} on $\SPrsh(\C)$ has the following structure: the {\it generating trivial cofibrations} will be those maps of the form 
$\Lambda [n] _+ \wedge  \rmU_+ \ra \Delta [n] _+ \wedge  \rmU_+$, $  \rmU \eps \C$, $ n \ge 1$
where $\Lambda [n]$ denotes a {\it horn} of $\Delta[n]$. If $K$ is a simplicial set, $ K_+ \wedge  \rmU_+$ is the 
simplicial presheaf defined by $(K _+ \wedge  \rmU _+)_n = (\sqcup _{\alpha_n \eps K_n}  \rmU)_+$ with the structure maps induced from those of $K$. 
The {\it generating cofibrations} are those maps of the form: $\delta \Delta [n] _+\wedge  \rmU _+ \ra \Delta [n]_+ \wedge  \rmU_+$, with $ \rmU \eps \C$, $n \ge 1$. 
The {\it weak-equivalences} will be those maps $f:\oF' \ra \oF$ of simplicial presheaves that induce weak-equivalence stalk-wise and the {\it fibrations} are defined
by the right-lifting-property with respect to trivial cofibrations. Then this model category is a cofibrantly generated simplicial model category which is 
also proper and cellular.
\vskip .3cm
Next we discuss the {\it flasque} model structure. Let $\rmX \eps \C$ and let $\U =\{ \rmU_{\alpha} \ra \rmX| \alpha =1, \cdots, n\}$ denote a finite collection of monomorphisms in $\C$.
Then one defines $U \U =  \rmU_{\alpha=1}^n  \rmU_{\alpha}$ to denote the co-equalizer of the diagram:
$\sqcup_{\alpha, \beta}  \rmU_{\alpha}{\underset {\rmX} \times}  \rmU_{\beta} \stackrel{\ra}{\ra} \sqcup_{\alpha}  \rmU_{\alpha}$ where the map in the top row (bottom row) is
the projection to the first (second, \res) factor. Given a map $f: F \ra G$ in $\SPrsh(\C)$ and a map $g: K \ra L$ of pointed simplicial sets, 
the pushout-product $f \Box g$ is the map $G \wedge K \bigvee _{F \wedge K} F \wedge L \ra G \wedge L$. Then we define the {\it flasque} model
structure on $\SPrsh(\C)$ to have the following structure: the {\it generating trivial cofibrations} are of the form $f \Box g$, where 
$f: U\U \ra  \rmU$ in $\SPrsh(\C)$ and $g: \Lambda [n] \ra \Delta [n]$, $n \ge 1$ and the generating cofibrations are of the form $f \Box g$, where
$f$ is as above and $g: \delta \Delta [n] \ra \Delta [n]$, $n \ge 1$.  The weak-equivalences are those  maps $f:F' \ra \oF$ of simplicial presheaves that induce weak-equivalence stalk-wise and the {\it fibrations} are defined
by the right-lifting-property with respect to trivial cofibrations. Then this model category is a a cofibrantly generated simplicial model category which is 
also left-proper and cellular.
\vskip .3cm \noindent
The weak-equivalences in both of the above model structures will often be referred to as {\it local weak-equivalences}. In \cite{Bl} and \cite{Isak}, what is
considered is the unpointed case; but their results readily extend to the pointed case as shown in \cite[Proposition 1.1.8 and Lemma 2.1.21]{Hov-1}.
The main use of the above model structures is to be able to prove the following result, which is used in the body of the paper. 
\begin{proposition}
 \label{sect.wise.weq}
Suppose $f:\E' \ra \E$ is a local weak-equivalence between fibrant objects in $\SPrsh(\C)$ provided with any of the above model structures. Then
$\Gamma ( \rmU, \f)$ is also a weak-equivalence for any $ \rmU \eps \C$. (Maps $f$ as above, for which $\Gamma ( \rmU, \f)$ is a weak-equivalence for every $ \rmU \eps \C$
will be called section-wise weak-equivalences.)
\end{proposition}
\begin{proof} We will first prove this result under the assumption that $f$ is also a fibration. Let $* = \Delta[0]$, let $n \ge 1$ be any integer and let $ \rmU \eps \C$. 
In this case we observe that it suffices to show that
 both the diagrams
\vskip .3cm
\xymatrix{{* } \ar@<1ex>[r] \ar@<1ex>[d] & {\Gamma ( \rmU, \E')} \ar@<1ex>[d] && {\delta \Delta [n]} \ar@<1ex>[r] \ar@<1ex>[d] & {\Gamma ( \rmU, \E')} \ar@<1ex>[d] \\
{\delta \Delta [n]} \ar@<1ex>[r] & {\Gamma ( \rmU, \E)} &\mbox { and } &{\Delta [n]} \ar@<1ex>[r] & {\Gamma ( \rmU, \E)}}
\vskip .3cm \noindent
have a lifting from the left-bottom corner to the right-top corner. To see this, observe that a class in $\pi_{n-1}(\Gamma ( \rmU, K)$ corresponds to a
map $\delta \Delta [n] \ra \Gamma ( \rmU, K)$ for any fibrant object in $K$ in $\SPrsh(\C)$. Therefore, the lifting in the first square shows that the
induced map $\pi_{n-1}(\Gamma ( \rmU, \f)$ is surjective. A lifting in the second square then shows that the induced map $\pi_{n-1}(\Gamma ( \rmU, \f)$ is 
injective as well. 
\vskip .3cm
Next assume that $f$ is a fibration as well. Then the two diagrams above correspond to the two diagrams:
\vskip .3cm
\xymatrix{{*_+ \wedge   \rmU_+} \ar@<1ex>[r] \ar@<1ex>[d] & {\E'} \ar@<1ex>[d] && {\delta \Delta [n] _+ \wedge  \rmU_+} \ar@<1ex>[r] \ar@<1ex>[d] & { \E'} \ar@<1ex>[d] \\
{\delta \Delta [n]_+ \wedge  \rmU_+} \ar@<1ex>[r] & { \E}  &\mbox { and } & {\Delta [n]_+ \wedge  \rmU_+ } \ar@<1ex>[r] & { \E}}
\vskip .3cm \noindent
The required liftings then correspond to liftings in the above squares, i.e., maps from the left-bottom corner to the right-top corner. But these
exist since $f$ is a trivial fibration and both the maps $* _+ \wedge  \rmU_+ \ra \delta \Delta[n]_+ \wedge  \rmU_+$ and $ \delta \Delta [n]_+ \wedge  \rmU_+ \ra 
\Delta [n] _+ \wedge  \rmU_+$
are cofibrations both in the local projective model structure and the local flasque model structure.
\vskip .3cm
Finally we consider the general case where $f$ is required to be only a map that is a weak-equivalence and both $\E'$ and $\E$ are required to be fibrant in either
of the above model structures. This then follows by invoking what is called Ken Brown's lemma: see 
\cite[Lemma 1.1.12]{Hov-1}. We will provide some
details on this for the convenience of the reader. Let $f:\E' \ra \E$ be the given map.  Form the  product $\E' \times \E$ with the projection to the
 factor $\E'$ ($\E$) denoted $p_1$ ($p_2$, \res). The maps $f:\E' \ra \E$ and the identity $id:\E' \ra \E'$ induce a map $\E' \ra \E' \times \E$. Factor this map
as the composition of a trivial cofibration $i$ followed by a fibration $q$. Then $p_1 \circ q \circ i =id_{\E'}$ and $p_2 \circ q \circ i =f$:
so $p_1 \circ q$ and $p_2 \circ q$ are both weak-equivalences as well as fibrations. Therefore, by what is proven above, 
$\Gamma ( \rmU, p_1 \circ q)$ and $\Gamma ( \rmU, p_2 \circ q)$ are both weak-equivalences for any $  \rmU \eps \C$. But $\Gamma ( \rmU, id) = 
\Gamma ( \rmU, p_1 \circ q \circ i) = \Gamma ( \rmU, p_1 \circ q) \circ \Gamma ( \rmU, i)$ so that $\Gamma ( \rmU, i)$ is a weak-equivalence. Therefore,
$\Gamma ( \rmU, \f) = \Gamma ( \rmU, p_2 \circ q \circ i) = \Gamma ( \rmU, p_2 \circ q) \circ \Gamma ( \rmU, i)$ is also a weak-equivalence. This completes the proof 
of the proposition. 
\end{proof}
\begin{proposition} 
\label{Godement.is.fibrant}
Assume the site $\C$ has enough points. For an $F \eps \SPrsh(\C)$, let $\G F = \holimD \{G^nF|n\}$ denote the simplicial presheaf
 defined by the {\it Godement resolution} $\{G^nF|n \}$ which is a cosimplicial object of $\SPrsh(\C)$ augmented by $\oF$. Then the obvious maps
$\Gamma ( \rmU, \G F) \ra *$ is a fibration for any $ \rmU \eps \C$ and $\G \oF$ is fibrant in 
the local projective  model structure if $\oF$ is stalk-wise fibrant. Therefore, if $f:\oF' \ra \oF$ is a stalk-wise weak-equivalence in $\SPrsh(\C)$ between pointed simplicial presheaves 
that are stalk-wise fibrant, then $\Gamma ( \rmU, \G \f)$ is a weak-equivalence for every $ \rmU \eps \C$.
\end{proposition}
\begin{proof} In order to prove the first statement, it suffices to show that one has liftings in the diagrams (from the left bottom corner to the right top corner):
\vskip .3cm
\xymatrix{{\Lambda [n]} \ar@<1ex>[r] \ar@<1ex>[d] & {\Gamma ( \rmU, \G F)} \ar@<1ex>[d] && {\Lambda [n]_+ \wedge  \rmU_+} \ar@<1ex>[r] \ar@<1ex>[d] & {\G F} \ar@<1ex>[d]\\
{\Delta [n]} \ar@<1ex>[r]& {*} &\mbox{   and   } & {\Delta [n]_+ \wedge  \rmU_+} \ar@<1ex>[r] & {*}}
\vskip .3cm \noindent
Clearly a lifting in the second diagram implies a lifting in the first diagram, so that it suffices to establish a lifting in the second diagram. 
Next consider the model structure on cosimplicial objects of $\SPrsh(\C)$ defined by viewing cosimplicial objects as diagrams of type $\Delta$ 
and invoking \cite[Theorem 11.6.1]{Hirsch}: recall the fibrations (weak-equivalences) in this model structure are those maps 
$f:\rmX^{\bullet} \ra \rmY^{\bullet}$ which are fibrations (weak-equivalences) for each fixed cosimplicial index $n$ and cofibrations are generated
by the free diagram construction applied to the generating cofibrations in $\SPrsh(\C)$. 
Making use of the
definition of the homotopy inverse limit, this amounts to obtaining lifting in the following diagram:
\vskip .3cm
\xymatrix{{\Lambda [n]_+ \wedge  \rmU_+ \wedge B(\Delta \downarrow - )^{op}} \ar@<1ex>[r] \ar@<1ex>[d] & {\{G^nF|n\}} \ar@<1ex>[d] \\
{\Delta [n]_+ \wedge  \rmU_+ \wedge B(\Delta \downarrow -)^{op}} \ar@<1ex>[r] & {*}}
\vskip .3cm \noindent
Now $B(\Delta \downarrow -)^{op}$ is a cofibrant object in this model structure on cosimplicial objects of $\SPrsh(\C)$ so that the 
left-vertical map in the above diagram is a trivial cofibration. Therefore, the required lifting exists, thereby proving the first statement.
The second statement then follows from Proposition ~\ref{sect.wise.weq}.
\end{proof}
\vskip .3cm
\begin{proposition} Let $\phi: \C \ra \D$ denote a morphism of sites. Assume $\SPrsh(\C)$ and $\SPrsh(\D)$ are both provided with the 
local projective model structures. Then $\phi_*$ preserves fibrations and trivial fibrations. It also preserves weak-equivalences between
fibrant objects. In particular, if $f: \oF' \ra \oF$ is a stalk-wise weak-equivalence between objects in $\SPrsh(\C)$ that are fibrant stalk-wise, then 
$\Gamma (\rmV, \phi_*(\G \f))$ is a weak-equivalence for every $\rmV \eps \D$.
 \end{proposition}
\begin{proof} Now consider the squares:
\vskip .3cm
\xymatrix{{\Lambda [n]_+ \wedge \rmV_+} \ar@<1ex>[r] \ar@<1ex>[d] & {\phi_* (F')} \ar@<1ex>[d]^{\phi_*(f)} && {\delta \Delta [n]_+ \wedge \rmV_+} \ar@<1ex>[r] \ar@<1ex>[d] & {f_*( F')} \ar@<1ex>[d]^{\phi_*(f)}\\
{\Delta [n] _+ \wedge \rmV_+} \ar@<1ex>[r]& {\phi_*(F)} &\mbox{   and   } & {\Delta [n]_+ \wedge \rmV_+} \ar@<1ex>[r] & {\phi_*(F)}}
\vskip .3cm \noindent
In order to prove the first statement, it suffices to show that if $f$ is a fibration (trivial fibration) then there is a lifting in the first (second, \res) square.
By adjunction, the above squares correspond to:
\vskip .3cm
\xymatrix{{\Lambda [n]_+ \wedge \phi^{-1}(\rmV)_+} \ar@<1ex>[r] \ar@<1ex>[d] & {F'} \ar@<1ex>[d]^{f} && {\delta \Delta [n]_+ \wedge \phi^{-1}(\rmV)_+} \ar@<1ex>[r] \ar@<1ex>[d] & { F'} \ar@<1ex>[d]^{f}\\
{\Delta [n] _+ \wedge \phi^{-1}(\rmV)_+} \ar@<1ex>[r]& {F} &\mbox{   and   } & {\Delta [n]_+ \wedge \phi^{-1}(\rmV)_+} \ar@<1ex>[r] & {F}}
\vskip .3cm \noindent
A lifting exists in the first square (second square) when $f$ is a fibration (trivial fibration, \res). Now use the adjunction one more to carry over
these liftings to the liftings in the original squares. This proves the first statement. The second statement then follows by invoking
Ken Brown's lemma: see  \cite[Lemma 1.1.12]{Hov-1}. The proof of the last statement is now clear.
\end{proof}

\vskip .3cm
Next let $\X = \{\rmX_{\a}|\a \eps \A\}$ denote a pro-object of $Sch(S)$. We will provide the site $\Top(\X)$ as well as the fibered site $\{\Top (\rmX_{\a})|\a \eps \A\}$
with the local projective model structures. Accordingly the generating cofibrations (generating trivial cofibrations) in $Top(\X)$
are  given by $\{\delta \Delta[n] _+ \wedge \U_+ \ra \Delta [n]_+ \wedge \U_+|n \ge 1, \U \eps \Top (\X)\}$ ($
\{ \Lambda[n] _+ \wedge \U_+ \ra \Delta [n]_+  \wedge \U_+|n \ge 1, \U \eps \Top(\X)\}$, \res). The generating cofibrations (generating trivial cofibrations)
in the fibered site $\{\Top(\rmX_{\a})|\a \eps \A \}$ are given by $\{\delta \Delta[n] _+ \wedge  \rmU_{+} \ra \Delta [n]_+ \wedge  \rmU_+|n \ge 1,  \rmU \eps \Top (\rmX_{\a}), 
\mbox{ for some a}\}$ ($\{ \Lambda[n] _+ \wedge  \rmU_+ \ra \Delta [n]_+ \wedge  \rmU_+|n \ge 1,  \rmU \eps \Top(\rmX_{\a}), \mbox{ for some a}\}$, \res).
\vskip .3cm
\begin{proposition}
 \label{Phi.preserves.fibrations}
Assume the above situation. Then the functor $\Phi:   \{\SPrsh(\Top(\rmX_{\a}))|\a \eps \A \} \ra \SPrsh(\Top(\X))$ preserves fibrations and trivial fibrations.
 It sends local weak-equivalences between fibrant objects to section-wise weak-equivalences. In particular, if $f:\oF' =\{\oF'_{\a}| \a \eps \A\} 
\ra F = \{\oF_{\a}|\a \eps \A\}$ is a stalk-wise weak-equivalence, then the induced map $\Gamma (\U, \G f)$ is also a weak-equivalence for every $\U \eps \Top(\X)$.
If $f: \X \ra \Y $ is a map of pro-objects in $\Sch(S)$,  the functor $\Phi$ commutes with $f_*$.
\end{proposition}
\begin{proof} 
 Let $f:\oF' \ra \oF$ denote a map in $ \{\SPrsh(\Top(\rmX_{\a}))|\a \eps \A \}$ which is a fibration, i.e., each $f_{\a}:\oF'_{\a} \ra \oF_{\a}$ is a 
fibration in $\SPrsh(\Top(\rmX_{\a}))$. To prove the first statement, it suffices to show that one has a lifting in the squares:
\be \begin{equation}
\label{req.lift}
\xymatrix{{\Lambda [n]_+ \wedge \U_+} \ar@<1ex>[r] \ar@<1ex>[d] & {\Phi(\oF')} \ar@<1ex>[d] && {\delta \Delta [n]_+ \wedge \U_+} \ar@<1ex>[r] \ar@<1ex>[d] & {\Phi(\oF')} \ar@<1ex>[d] \\
{\Delta [n] _+ \wedge \U_+} \ar@<1ex>[r] & {\Phi( \oF)} &\mbox { and } &{\Delta [n] _+ \wedge \U_+} \ar@<1ex>[r] & {\Phi( \oF)}}
\end{equation} \ee
\vskip .3cm \noindent
These liftings correspond to liftings in the squares:
\vskip .3cm
\xymatrix{{\Lambda [n]_+ } \ar@<1ex>[r] \ar@<1ex>[d] & {\colima \Gamma ( \rmU_{\a}, \oF'_{\a})} \ar@<1ex>[d] && {\delta \Delta [n]_+ } \ar@<1ex>[r] \ar@<1ex>[d] & {\colima \Gamma ( \rmU_{\a}, \oF'_{\a})} \ar@<1ex>[d] \\
{\Delta [n] _+ } \ar@<1ex>[r] & {\colima \Gamma ( \rmU_{\a},  \oF_{\a})} &\mbox { and } &{\Delta [n] _+ } \ar@<1ex>[r] & {\colima \Gamma ( \rmU_{\a},  \oF_{\a})}}
\vskip .3cm \noindent
Now the smallness of finite simplicial sets shows that the above squares correspond to squares:
\vskip .3cm
\xymatrix{{\Lambda [n]_+ } \ar@<1ex>[r] \ar@<1ex>[d] & {\Gamma ( \rmU_{a_0}, \oF'_{a_0})} \ar@<1ex>[d] && {\delta \Delta [n]_+ } \ar@<1ex>[r] \ar@<1ex>[d] & { \Gamma ( \rmU_{a_0}, \oF'_{a_0})} \ar@<1ex>[d] \\
{\Delta [n] _+ } \ar@<1ex>[r] & { \Gamma ( \rmU_{a_0},  \oF_{a_0})} &\mbox { and } &{\Delta [n] _+ } \ar@<1ex>[r] & { \Gamma ( \rmU_{a_0},  \oF_{a_0})}}
\vskip .3cm \noindent
for some $a_0 \eps \A$. Clearly these squares correspond to:
\vskip .3cm
\xymatrix{{\Lambda [n]_+ \wedge  \rmU_{a_0}} \ar@<1ex>[r] \ar@<1ex>[d] & {\oF'_{a_0}} \ar@<1ex>[d] && {\delta \Delta [n]_+ \wedge  \rmU_{a_0}} \ar@<1ex>[r] \ar@<1ex>[d] & {  \oF'_{a_0}} \ar@<1ex>[d] \\
{\Delta [n] _+ \wedge  \rmU_{a_0}} \ar@<1ex>[r] & {   F_{a_0}} &\mbox { and } &{\Delta [n] _+  \wedge  \rmU_{a_0}} \ar@<1ex>[r] & { \oF_{a_0}}}
\vskip .3cm \noindent
The left vertical map in the first square is a trivial cofibration in $\SPrsh(\Top(\rmX_{a_0}))$ and the left vertical map in the second square is
a cofibration in $\SPrsh(\Top(\rmX_{a_0}))$ when they are provided with the local projective model structures. Therefore, the required lifting exists
in the first square when $f$ is a fibration and the required lifting exists in the second square when $f$ is a trivial fibration. Tracing back,
and using the fact that $\A$ is cofiltered, one may see that these show the required liftings exist in ~\eqref{req.lift}. 
\vskip .3cm
Next suppose $f: \oF' \ra \oF$ is a stalk-wise weak-equivalence between fibrant objects. Recall each $\oF'_{\a}$ and $\oF_{\a}$ is fibrant in the local projective
model structure on $\Top (\rmX_{\a})$. Therefore, Proposition ~\ref{sect.wise.weq} shows that for each $ \rmU_{\a} \eps \Top(\rmX_{\a})$, $ \Gamma ( \rmU_{\a}, f_{\a})$ is 
a weak-equivalence. If $\U=\{ \rmU_{\a}| \a \eps \A\} \eps \Top(\X)$, then $\Gamma (\U, f) = \colima \Gamma ( \rmU_{\a}, f_{\a})$ so that $\Gamma (\U, f)$ is also
a weak-equivalence. The last but one statement is now clear in view of Proposition ~\ref{Godement.is.fibrant}. 
\vskip .3cm
Recall that $\Phi(\oF) = \colima ()_{\a}(\oF_{\a})$ if $\oF = \{\oF_{\a}|\a \eps \A\}$. Therefore, 
\vskip .3cm
$\Gamma (\U, \f_*(\Phi(\oF))) = \Gamma (\f^{-1}(\U), \Phi(\oF)) =
\colima \Gamma (\f^{-1}(\U), ()_{\a}(\oF_{\a})) =\colima \Gamma (\f_{\a}^{-1}( \rmU_{\a}), \oF_{\a})$
\vskip .3cm
$ =\colima \Gamma ( \rmU_{\a}, \f_{a*}(\oF_{\a})) = \colima \Gamma (\U, ()_{\a}(\f_{a*}(\oF_{\a})))
=\Gamma (\U, \Phi(\{\f_{a*}(\oF_{\a})| \a \eps \A\}))$. 
\vskip .3cm \noindent
This proves the last statement and completes the proof.
\end{proof}
\subsubsection{Hypercohomology spectra and the right derived functor of the direct image functor}
\label{hypercoh.pro.schms}
Let $\X$ denote a pro-scheme.
In view of the above observation, one may define both of these making use of the Godement resolutions. Alternatively one
may choose a model structure where the fibrations and weak-equivalences are defined to be those which are object-wise, i.e., 
$\E \ra B$ is a fibration (weak-equivalence) if and only if $\Gamma (\U, \E) \ra \Gamma (\U, B)$ is a fibration (weak-equivalence, \res) for every 
$\U \eps To(\X)$ and
cofibrations defined by the lifting property. Given a presheaf of spectra $\E$ on $Top(\X)$ or on $\{Top(\rmX_{\a})|\a \eps \A\}$, we let
$\G \E$ denote such a fibrant replacement.
Therefore, one defines
\vskip .3cm
$\H(\X, \E) = \Gamma (\X, \G \E)$.
\vskip .3cm \noindent
If we start with an $\E' =\{\E'_{\a} = \mbox {a presheaf of spectra on } Top(\rmX_{\a})|\a \eps \A\}$, then 
$ \rmR\Phi(\E')= \Phi(\G \E')$. Then 
one readily observes that $\H(\X, \rmR\Phi(\E')) = \colima \H(\rmX_{\a}, \E'_{\a})$. In fact one may see that the functor $\Phi$ preserves
object-wise fibrations and weak-equivalences. So $R\Gamma (\U, \rmR \Phi(\E')) $ identifies up to weak-equivalence with
$\Gamma(\U, \rmR\Phi(\E'))$, for any $\U  \eps Top(\X)$.
\vskip .3cm
Next let $f: \X= \{\rmX_{\a}| \a \eps \A\} \ra \Y= \{\rmY_{\a}|\a \eps \A\}$ denote a level map of pro-schemes. Then we let
\vskip .3cm
$\rmR \f_*(\E) = \f_*(\G \E)$.
\vskip .3cm \noindent
In case $\E= \Phi(\E')$ as above, then $\rmR f_*(\Phi(\E')) $ identifies up to weak-equivalence with 
\[\f_* (\rmR\Phi(E')) = \colima \{\rmR \f_{a*}(E'_{\a})|\a \eps \A\}.\]
\section{Comparison with  equivariant cohomology defined by the tower construction}
In this section we proceed to obtain a comparison of the generalized Borel-style equivariant cohomology theories defined
in this paper for pro-group actions with a corresponding theory defined in \cite{C13} and with a slightly different variant considered
 in \cite{CJ26-2}.
Let $\oF$ denote a field containing the algebraically closed field $\ok$ and let $\rmG=\rmG_{\oF}$ denote the absolute Galois group of $\oF$.
Then it is proven in \cite[Proposition 3.2]{C13} that it is {\it totally torsion free}, i.e., for every closed subgroup 
$K \subseteq \rmG_{\oF}$, the abelianization $K^{ab}$ is torsion free.
\vskip .2cm
At this point there are two different contexts to consider.
\vskip .1cm
\begin{enumerate}[\rm(i)]
 \item {\it A basic assumption 
here is that}
\be \begin{equation}
     \label{basic.Carl.assumption}
\mbox{ the the profinite Galois group $\rmG_{\oF}$ 
is a pro-$\ell$-group, with $\ell$ a prime different from $char(\ok)$.} 
\end{equation} \ee
A monomial representation (also called an affine $\ell$-adic representation in 
\cite[sections 4, 5]{C13}) $\rho$ of $\rmG$ is a continuous homomorphism 
$\rho: \rmG \ra \Sigma_n \ltimes \Z_{\ell^n}$. $n$ is the dimension of the representation. For each pair $(\rho, m)$ where
$\rho$ is a monomial representation as above and $m \eps \NN$, we let $\E{\rmG}^C_{\rho, m}$ denote the pro-scheme
constructed in \cite{C13}. Recall $\E{\rmG}^C_{\rho, m}$ is a tower indexed by $i \eps \NN$ so that
\be \begin{equation}
     (\E{\rmG}^C_{\rho, m})(i) = (({\mathbb G}_m)^n)^m
    \end{equation} \ee
and where the structure map ${\E{\rmG}^C_{\rho,m}}(i) \ra {\E{\rmG}^C_{\rho,m}}({i-1})$ is the
$\ell$-th power map. The action of the Galois group $\rmG$ on this tower is induced from  the diagonal action of 
$\Sigma_n \ltimes {\mathbb Z} _{\ell^n}$ through
representations of $\rmG$ in $\Sigma_n \ltimes {\mathbb Z}_{\ell^n}$. The group $\Sigma_n$ acts by permuting the $n$-factors of ${\mathbb G}_m^n$ while
$\Z_{\ell^n}$ acts on a factor of the tower formed by the ${\mathbb G}_m$ as follows. One first fixes a system of primitive $\ell^i$-th roots
of unity in the field $\ok$, for all $i \ge 1$. If ${\mathbb G}_m(i)$ denotes the $i$-th term of the
tower of ${\mathbb G}_m$, then one lets $\Z_{\ell}$ act on it through its quotient $\Z/{\ell^i}$ by multiplication by the
$\ell^i$-th primitive roots of unity in $\ok$. We let 
\be \begin{equation}
     \label{EGC.1}
\E\rmG ^C = \Pi_{\rho, m} \E\rmG ^C_{\rho,m}
    \end{equation} \ee
\item A {\it second and distinct} set-up is where one assumes the base field $k$ {\it is of characteristic $0$}. In this framework,
one considers the {\it Puisceux schemes} $\{{\mathcal P}^i|i\}$ defined in \cite[section 4]{CJ26-2}. In this case we let
\be \begin{equation}
     \label{EGC.2}
\E\rmG ^C = \{{\mathcal P}^i_n({\mathbb T}^i)|n, i\}
    \end{equation} \ee
where ${\mathbb T}^i$ is the complement in ${\mathbb A}^i$ of all the coordinate hyperplanes.
\end{enumerate}
\vskip .3cm
Then a basic result we need is the following result. 
Let $\rmG= \{\rmG_{\a}|\a \eps \A\}$   with $\rmG_{\a}$ running over finite 
quotient groups of $\rmG$. Then there exists a pro-scheme $\{\cY_{\a}| \a \eps \A\}$ provided with an action on 
$\cY_{\a}$ by $\rmG_{\a}$, and with all the actions compatible as $a $ varies over $\A$. 
Moreover  each $\cY_{\a}$ is smooth and the action of $\rmG_{\a}$ on $\cY_{\a}$ is free and the inverse system $\{\cY_{\a}|\a \}$ represents
$\E\rmG ^C$ in both cases. This is proved in \cite[Proposition 5.9]{C13} for the first case and in \cite[Theorem 8.1 and Remark 8.1]{CJ26-2} for the second case.
Therefore, one obtains the 
weak-equivalences (in both cases):
\be \begin{align}
     \label{Carlsson.eq}
\Map((\oSpec \, \bar \oF) \times_{\rmG_{\oF}}  \E{\rm G}_{\oF}^C, \bK_{\ell}) &\simeq \colima \Map((\oSpec \, \oF_{\a})_{\a}\times _{\rmG_{\a}} \cY_{\a}, \bK_{\ell})\\
\Map((\oSpec \, \ok) \times_{\rmG_{\oF}}  \E{\rm G}_{\oF}^C, \bK_{\ell}) &\simeq \colima \Map((\oSpec \, \ok)_{\a}\times _{\rmG_{\a}} \cY_{\a}, \bK_{\ell}) \notag
\end{align} \ee
\vskip .3cm \noindent
where the left-hand-side of the first equation (second equation) is defined as 
$\colimb \Map((\oSpec \, \oF_{\b}) \times _{\rmG_{\b}} \E{\rm G}_{\oF}^C(b), \bK_{\ell})$  ($\colimb \Map((\oSpec \, \ok) \times_{\rmG_{\b}} \E{\rm G}_{\oF}^C(b), \bK_{\ell})$, \res). 
$\rmG_{\b}$ is a finite quotient of $\rmG_{\oF}$ that acts on the $b$-th stage of the tower $\E{\rm G}_{\oF}^C$ as well as on the finite normal 
extension $\oF_{\b}$. $\bK_{\ell}$ 
 denotes any of the  spectra considered in Theorem ~\ref{main.thm.1} with $\E=\bK$ being the spectrum representing algebraic K-theory
on $\Sm/k$.
\vskip .3cm
In this situation, we may start with a closed immersion $\rmG_{\a} \ra {\mathfrak G}_{\a}$ for each
$\a \eps \A$, with ${\mathfrak G}_{\a}$ being a linear algebraic group. Then we may form an inverse system of 
algebraic groups by replacing each ${\frG}_{\a}$ with the the finite product $\bG_{\a}=\Pi_{b \le a}{\mathfrak G}_{\b}$ where the
structure maps are induced by the obvious projection maps: see the construction in ~\ref{compat.inverse.systems} and the discussion 
following ~\ref{bfG.def} for more details. Let $\{s(a)|\a \eps \A\}$ denote a non-decreasing sequence with each $s(a)$ a 
non-negative integer. (Recall $\K$ denotes the directed set of all such sequences.)
\vskip .3cm
Next suppose $\{\cX_{\a} | \a \eps \A\}$ is an inverse systems 
of smooth schemes of finite type over $\ok$ so that each 
$\cX_{\a}$  is provided with an action by the  group $G_{\a}$ and
the these actions are compatible. 
\begin{definition}
\label{Hypercoho.Carl}
Assume the above situation. In view of the weak-equivalence ~\eqref{Carlsson.eq}, for any spectrum $\E \eps \Spt_{S^1}(k)$, we let:
 \[\Map(\cX \times_{\rmG_{\oF}}  \E{\rm G}_{\oF}^C, \E) =\colima \Map(\cX_{\a}\times _{G_{\a}} \cY_{\a}, \E).\]
\end{definition}
\vskip .3cm \noindent
Now Theorem ~\ref{uniqueness.pro.grp.actions}(ii) readily provides the following comparison result.
\begin{theorem} 
\label{main.thm.3}
Assume that $\oF$ is a field containing the algebraically closed field $\ok$, let $\rmG=\rmG_{\oF}$ denote the absolute Galois group of $\oF$. 
Assume that one of the following conditions holds: 
\begin{enumerate}[\rm(i)]
\item $\rmG$ is a free pro$-\ell$ group with $\ell$ prime to the characteristic of $\ok$, or
\item the field $k$ is of characteristic $0$.
\end{enumerate}
\vskip .3cm
Then we obtain the string of weak-equivalences:
\[\holims \colima \H(\E\bG^{gm,s}_{\a}\times _{\rmG_{\a}}(\cX_{\a} \times \cY_{\a}), \bK_{\ell}) \simeq \holimt \holims \colima \H(\E\bG^{gm,s}_{\a}\times _{\rmG_{\a}}(\cX_{\a} \times \cY_{\a}), P_{\le t} \bK_{\ell}) \]
 \vskip .3cm
\[\simeq \holimt \colima \H(\cX_{\a}\times _{\rmG_{\a}} \cY_{\a}, P_{\le t} \bK_{\ell}) \simeq \colima \H(\cX_{\a}\times _{\rmG_{\a}} \cY_{\a}, \bK_{\ell}) \]
\vskip .3cm
\[ = \Map(\cX \times_{\rmG_{\oF}}  \E{\rm G}_{\oF}^C, \bK_{\ell}).\]
\end{theorem}
\vskip .3cm
The last comparison result along with the rigidity theorem, Theorem ~\ref{main.thm1.1} now provides rigidity for mod-$\ell$ or
$\ell$-primary Borel-style equivariant K-theory defined in \cite{C13} and also for the corresponding theory considered in \cite{CJ26-2}.
\begin{theorem}
 \label{rigidity.Carlsson.style}
 Let $\rmG_{\oF} = \{G_{\a}|\a\}$.
Assume as in the above situation that $ \cX=\{\cX_{\a} | \a \eps \A\}$ is an inverse systems
of smooth schemes of finite type over $\ok$ so that each
$\cX_{\a}$  is provided with an action by the  group $G_{\a}$ and
the these actions are compatible.
 Then we obtain the weak-equivalence:
\be \begin{multline}
     \begin{split}
\Map(\oSpec \, k \times_{\rmG_{\oF}}( \cX \times \E{\rm G}_{\oF}^C),  \bK_{\ell})   \ra \Map(\oSpec \, \bar \oF \times_{\rmG_{\oF}} (\cX \times \E{\rm G}_{\oF}^C), \bK_{\ell})
\end{split}
\end{multline} \ee
\vskip .3cm \noindent
\end{theorem}
\begin{proof} Let $\bY=\{\cY_{\a}|\a \eps \A \}$ denote the inverse system of smooth schemes obtained from the tower
$\E{\rm G}_{\oF}^C$. Then the last theorem identifies the left-hand-side (right-hand-side) up to weak-equivalence with
\[\holims \colima \H(\E\bG^{gm,s}_{\a}\times _{\rmG_{\a}}( \oSpec \, \ok \times (\cX_a \times \cY_{\a}), \bK_{\ell}) \, ( \holims \colima \H(\E\bG^{gm,s}_{\a}\times _{\rmG_{\a}}(  \oSpec \,  \oF_{\a} \times \cX_a \times \cY_{\a}), \bK_{\ell}), {\rm \res}).\]
 Then the assertion that the map represented
by the arrow is a weak-equivalence is Theorem ~\ref{main.thm.1}.
\end{proof}
\vskip .3cm \noindent
{\bf Proof of Theorem ~\ref{main.thm.2}}. One may observe that the first (second) statement is the content of Theorem ~\ref{main.thm.3} (Theorem ~\ref{rigidity.Carlsson.style}, \res).
Taking the spectrum $\rmE = \bK$,  the last statement follows from the statement (ii) once one observes that $\rmG_{\oF}$ acts freely on $\oSpec \, \bar \oF$
so that $(\oSpec \, \bar \oF) \times _{\rmG_{\oF}} {\rm EG}^{C}_{\oF} \cong \oSpec \, \oF \times {\rm EG}^C_{\oF}$.



\end{document}